\documentclass[leqno,11pt,oneside]{amsart}

 \usepackage[english]{babel}

\usepackage{xcolor}
\usepackage{graphicx}

\usepackage{enumitem}
\usepackage{amsmath,amsfonts,amssymb,amsthm}

\newcommand{\sing}{{\rm Sing}}

\usepackage[english]{babel}

\newtheorem*{theorem*}{Theorem}

\newtheorem{teo}{Theorem}[section]
\newtheorem{prop}[teo]{Proposition}
\newtheorem{ddef}[teo]{Definition}
\newtheorem{example}[teo]{Example}

\newtheorem*{cor*}{Corollary}

\newtheorem*{remark}{Remark}
\newtheorem{assertion}{Assertion}

\newtheorem*{lem*}{Lemma}

\newtheorem*{teorA'}{Theorem A'}

\newtheorem*{fact*}{Fact}
\newtheorem{fact}{Fact}

\newtheorem{maintheorem}{Theorem}

\theoremstyle{definition}

\newcommand{\val}{{\rm Val}}

\newcommand{\C}{\mathbb{C}}
\newcommand{\CC}{\mathcal{C}}
\newcommand{\R}{\mathbb{R}}
\newcommand{\Pe}{\mathbb{P}}

\newcommand{\Z}{\mathbb{Z}}
\newcommand{\Q}{\mathbb{Q}}

\newcommand{\QQ}{{\bold  Q}}
\newcommand{\TT}{{\bold T}}
\newcommand{\DD}{{\bold D}}

\newcommand{\br}{{\mathcal B}}

\newcommand{\OO}{\mathcal{O}}
\newcommand{\A}{\mathcal{A}}

\newcommand{\D}{\mathcal{D}}
\newcommand{\E}{\mathcal{E}}
\newcommand{\F}{\mathcal{F}}

\newcommand{\LL}{{\mathcal L}}
\newcommand{\HH}{\mathcal{H}}
\newcommand{\W}{{\mathcal{W}}}

\newcommand{\dic}{\mathcal{D}^{\delta}}
\newcommand{\ndic}{\mathcal{D}^{\iota}}
\newcommand{\sepdic}{\mathcal{S}^{\delta}}
\newcommand{\sepndic}{\mathcal{S}^{\iota}}
\newcommand{\supp}{\mathcal{Z}}

\newcommand{\ic}{{\rm i}}

\newcommand{\cl}[1]{\mbox{$\mathcal{#1}$}}

\newcommand{\Iso}{{\rm Iso}}
\newcommand{\Dic}{{\rm Dic}}

\newcommand{\sep}{\mathcal{S}}

\DeclareMathOperator{\gr}{Gr}

\newcommand*\xbar[1]{ %
   \hbox{ %
     \vbox{%
       \hrule height 0.3pt 
       \kern0.35ex
       \hbox{%
         \kern-0.1em
         \ensuremath{#1}%
         \kern-0.1em
       }%
     }%
   }%
}

\newcommand*\xxbar[1]{%
   \hbox{%
     \vbox{%
       \hrule height 0.3pt 
       \kern0.4ex
       \hbox{%
         \kern-0.1em
         \ensuremath{#1}%
         \kern-0.1em
       }%
     }%
   }%
}

\makeindex

\begin{document}

\setcounter{section}{0}
\setcounter{teo}{0}
\setcounter{exe}{0}

\title{Logarithmic models and meromorphic functions in dimension two}

\author{Jane Bretas}
 \address{Departamento de Matem\'atica --- 
Centro Federal de Educa\c c\~ao Tecnol\'ogica de Minas Gerais}
\curraddr{Av. Amazonas, 7675 --- 30510-000
  --- Belo Horizonte, BRAZIL}
\email{janebretas@cefetmg.br}

\author{Rog\'erio   Mol }
\address{Departamento de Matem\'atica ---  Universidade Federal de Minas Gerais}
\curraddr{Av. Ant\^onio Carlos 6627 --- 31270-901 --- Belo Horizonte, BRAZIL.}
\email{rmol@ufmg.br}


\thanks{Second   author  partially financed by Pronex-Faperj}
\maketitle

\begin{abstract}
In this article we describe the construction of logarithmic models in both real and complex cases.
A logarithmic model is a germ of closed meromorphic 1-form with simple poles --- and
the analytic foliation defined by it --- produced upon
some specified geometric data: the structure of dicritical (non-invariant) components in the exceptional divisor of its reduction of singularities, a prescribed finite set of separatrices --- invariant analytic branches at the origin ---
and Camacho-Sad indices with respect to these separatrices. As an application, we use logarithmic models in order to construct real and complex germs of meromorphic
functions with a given indeterminacy structure and prescribed sets of zeroes and poles.
Also, in the real case, in the specific case where all trajectories accumulating at the origin are contained
in analytic curves, logarithmic models are used in order to build germs of analytic vector fields with a given Bendixson's sectorial decomposition
  of a neighborhood of $0 \in \R^{2}$ into
hyperbolic, parabolic and elliptic sectors.  As a consequence, we can produce real meromorphic functions with   prescribed sectorial decompositions.
\end{abstract}

\tableofcontents

 \medskip \medskip

\section{Introduction}
The fundamental  motivation for this article is the following question:

\medskip
\par \noindent
{\it
Are there germs of meromorphic functions, in the real or in the complex plane, with  a given indeterminacy structure   and
prescribed zeroes and poles?
}
\medskip

\noindent
By the indeterminacy structure of a meromorphic function  we mean the information given by the dicritical components of its resolution by a sequence of quadratic
blow-ups, i.e. the irreducible components of the exceptional divisor not   contained in a level set of the lifted function.
We prescribe zeroes and poles by providing a finite set $\sep$ of analytic branches, i.e. irreducible analytic curves at the origin.
In the complex case, we can give a positive  answer to this question  once we complete $\sep$ into a finite set
of branches satisfying some conditions of axiomatic nature, so that it forms a set of ``separatrices'' (Theorem \ref{teo-structure-meromporphic}).
In the real case, the answer is always positive, whoever $\sep$ is (Theorem \ref{teo-structure-meromporphic-real}).

We handle the above question in the broader context of the local theory of singular  holomorphic foliations in the complex plane, by  addressing the specific problem of building the so-called logarithmic models.
They consist in the construction of germs of singular  foliations of logarithmic type ---
those defined by  closed meromorphic $1$-forms with simple poles --- with a prescribed set of geometric data:
 dicritical components (non-invariant components of the exceptional
divisor in the reduction of singularities),
separatrices (invariant complex analytic branches), and Camacho-Sad indices. The latter are   local residue-type invariants,
associated with   pairs
singularity/separatrix, that play a significant role in the local topological characterization of the foliation and also impose
combinatorial restrictions along its reduction of singularities  \cite{camacho1982,linsneto1986,suwa1995}.
These data may originate, for instance, from   a   germ of singular complex analytic foliation $\F$ of \emph{generalized curve} type \cite{camacho1984}, which means
that, if $\pi$ is a reduction of singularities of $\F$ by  a sequence of
quadratic blow-ups, then the lifted foliation
$\pi^{*}\F$ has no saddle-node singularities, i.e. simple singularities with one zero eigenvalue (see Examples \ref{ex-foliation-structure} and \ref{ex-foliation-separatrices}).
In this case, a logarithmic model for $\F$ is
a germ of logarithmic foliation $\LL$ that, roughly speaking, has
the same reduction of singularities of $\F$, with dicritical  components positioned in the same places, having the same
 isolated separatrices   (those whose lift  by $\pi$ touch invariant components of the exceptional divisor) and the same Camacho-Sad indices.
All of this is considered   up to the existence of some extra  points, placed in the smooth part of dicritical components,
 where $\LL$ has a holomorphic first integral and   demands additional blow-ups in order to complete its reduction of
 singularities. These are called
 \emph{escape points}. They give rise to some isolated separatrices for $\LL$,  called accordingly \emph{escape separatrices}.
The logarithmic model is said to be \emph{strict} if there are no escape points, thus $\F$ and $\LL$ have the
same reduction of singularities and, hence,   coincident sets of isolated  separatrices.

Logarithmic models of complex foliations in dimension two were studied in the articles \cite{corral2003} and
\cite{canocorral2006}. 
In them, logarithmic foliations are built upon models provided by concrete foliations of generalized
curve type.
Our objective in this article is to make a  construction of general nature,  in  both real and complex  cases,
founded on data
in principle not coming from an actual foliation.
In fact, what we seek is a  method for producing examples of logarithmic   foliations
--- especially of those defined by meromorphic functions ---  satisfying  specific   properties
  described in terms of some  geometric data, as explained in the paragraph above.
This being said, our approach is  substantially different from that of the  two aforementioned papers,
 their main results being particular cases of our construction (see Examples \ref{ex-foliation-structure},  \ref{ex-foliation-separatrices}, \ref{ex-foliation-indices} and \ref{ex-foliation-quadruplet}).
It bears however some resemblance to the philosophy of \cite{corral2012},
where logarithmic models for real non-dicritical foliations are studied.

The core of our method is the consideration of blocks of information, with an axiomatic structure inspired in the properties of foliations  of generalized curve type,
  having, in principle, no connection with a concrete foliation. They are built upon
the various steps of a sequence of quadratic blow-ups $\pi$, that simulates the reduction of singularities of some hypothetical foliation.
The first of these blocks is formed by the exceptional divisors of $\pi$ and of its intermediate factors, having some distinguished irreducible components
that  play the role of dicritical components of the reduction divisor  of a singular foliation.
The second is a finite set $\sep_{0}$ of complex
analytic branches at $(\C^{2},0)$, along with their strict transforms by all  steps of $\pi$,
 that impersonate   separatrices of a foliation and are thus named this way.
 These two  blocks are called, respectively, \emph{dicritical structure} (Definition \ref{def-dic-structure}) and \emph{configuration of separatrices}
(Definition \ref{def-conf-separatrices}). Together
 they form a \emph{dicritical duplet}.
Then, we have a set of local complex invariants assigned to separatrices and non-dicritical components,
which emulate Camacho-Sad indices, called \emph{system of indices} (Definition \ref{def-system-indices}).
Its attachment to a dicritical duplet forms a \emph{dicritical triplet}.
Our goal is then  to assign, to  separatrices and non-dicritical components, invariants in $\C^{*}$ called \emph{residues}, to be assembled in
a fourth block called  \emph{system of residues} (Definition \ref{def-system-residues}), in
a consistent way with the information provided by a dicritical triplet (Definition \ref{def-consistent}), in order to form a
  \emph{dicritical quadruplet}. Simultaneously, we wish that  a logarithmic $1$-form with poles in $\sep_{0}$ and the given residues defines a germ
of singular foliation $\LL$ that is a \emph{logarithmic model}
for all blocks of data considered (Definition \ref{def-Q-logarithmic}).
In a nutshell, up to the existence of escape points and escape separatrices, this means that the reduction of singularities of $\LL$, its isolated separatrices, Camacho-Sad indices and residues are precisely those provided
by the dicritical quadruplet.
The existence of  consistent data of residues and of a logarithmic $1$-form with these properties is given by Theorem  \ref{teo-quadruplet},
which will be the main source for all other results  in this paper.
Its proof relies mostly on arguments of \cite{canocorral2006}, properly adapted to our objects.
 The axiomatic construction of our dicritical multiplets is carried out in Section
 \ref{section-dicritical-structure}
and in Section \ref{section-system-indices}, while the proof of Theorem \ref{teo-quadruplet} is
done in Section \ref{section-existence-quadruplet}.

Next, in  Section \ref{section-moromorphic-functions}, we state and prove Theorem  \ref{teo-structure-meromporphic},  which gives an answer to the complex version of the question that opens this article.
We start with a dicritical duplet containing the prescribed data of indeterminacy and branches of zeroes and poles.
Then, using   combinatorial tools of \cite{camacho1988},
we  produce a system of rational negative  indices   at the final level --- the one corresponding to the whole sequence of blow-ups
 $\pi$.
This condition, necessary for them to be  actual
Camacho-Sad indices of reduced models of  a meromorphic function,
 also happens to be sufficient. Indeed,
Theorem \ref{teo-quadruplet}  provides a system of rational residues, so that  the wished meromorphic function  arises in a straightforward way.

Then, in Section \ref{section-real-models}
and in Section \ref{section-escape}, we turn our attention to the real case,  proving the existence of   real logarithmic models  in Theorem \ref{teo-real-logmodel}.
Our point of departure is a sort of real  dicritical triplet  containing only real information (sequences of    quadratic blow-ups,   analytic branches and  indices, all of them real),
with a weakened axiomatic structure (Definition \ref{def-real-quasi-dicritical-triplet}). Its complexification   can be completed into a  dicritical triplet,
symmetric with respect to the involution induced
by the  complex conjugation.
The application of Theorem \ref{teo-quadruplet} then provides a system of real residues and a logarithmic $1$-form,
both symmetric with respect to the conjugation. The restriction of this complex logarithmic  $1$-form   to the real trace
of the complexification provides the
logarithmic model we seek.
It turns out that, in the real case, by adding some dicritical separatrices with appropriate residues,  real escape points can be eliminated
(Proposition \ref{prop-withour-real-escape}). This, combined with Theorem \ref{teo-real-logmodel},  can be used in order  to produce strict real logarithmic models,
as we do in Theorem \ref{teo-strict-logmodel-real}, where they are obtained for
 germs of real analytic foliations of real generalized curve type, i.e. whose real reductions of singularities do not contain   saddle-node singularities. Also,  by the application of
  the ideas leading to Theorem \ref{teo-real-logmodel},   we  obtain
Theorem  \ref{teo-structure-meromporphic-real}, which gives an affirmative answer to the real version of our opening question.

Finally, in the last part of the article, Section \ref{section-sectorial}, we  study   sectorial decompositions of germs of real analytic vector fields.
A classical result by I. Bendixson \cite{bendixson1901} asserts that a real analytic vector field, with an isolated
singularity at $0 \in \R^{2}$, of non-monodromic type, induces  a  decomposition of a small neighborhood of the origin  in
a finite number of sectors.
The term non-monodromic refers to the fact that there exists at least one characteristic orbit,
i.e. one accumulating at the origin
with a well defined tangent at the limit point.
Then, each  sector is limited by a pair of characteristic orbits and bears a classification into  hyperbolic, parabolic or elliptic,
according to the topological behavior of the orbits inside of it (see Section \ref{section-sectorial}
and also \cite{ilyashenko2008,roussarie2020} for a   detailed description).
The essential information concerning a sectorial decomposition --- a finite number of real analytic semicurves, with the mentioned accumulation properties, defining sectors classified into hyperbolic, parabolic or elliptic  --- would then define
a  \emph{sectorial model}.
A second guiding question, also a motivation for the development
of real logarithmic models  described above,
is posed in the following terms:
\medskip
\par \noindent
{\it
Given a sectorial model at $(\R^{2},0)$, is there a germ of real analytic vector field
that realizes it?
}
\medskip

\par
We answer this question in a particular case.
Within the family
of  non-monodromic  vector fields, we consider those
whose orbits accumulating at $0 \in \R^{2}$ are contained in the trace of
real analytic branches and call them $\ell$-analytic (Definition \ref{def-l-analytic}).
With them, we associate  more refined sectorial models, also called
$\ell$-analytic, that aggregate  infinitesimal information of
dicritical components, which in turn account for the existence of parabolic and elliptic sectors
(Definition \ref{def-sectorial-modeled}).
An almost immediate application of
the existence of real logarithmic models gives
Theorem  \ref{teo-sectorial-modeled}, which   asserts that any such a model is realized by a germ of $\ell$-analytic vector field, and
Theorem \ref{teo-sectorial-modeled-meromorphic}, which assures that it can be actually realized by a germ of real meromorphic function.
These two theorems and the constructive approach in their proofs provide, as an interesting application to mathematics design, a
systematic  method for the construction  of nice flower-type analytic (or algebraic) curves in the real plane (see Section \ref{sec-examples}).

\section{Dicritical structure and configuration of separatrices}

\label{section-dicritical-structure}

The problems we deal with  concern the construction of singular holomorphic foliations with   prescribed dicritical components in their reductions of singularities
and prescribed sets of invariant analytic  curves. As we commented in the introduction, we approach this by means of abstract structures
that handle these information in an axiomatic manner. In this section we introduce two of them:
dicritical structure and configuration of separatrices.

\subsection{Infinitesimal classes}

We start by establishing an intrinsic way to identify an irreducible component of the exceptional divisor of a sequence of quadratic blow-ups.
Let $\pi:(\tilde{M},\D) \to (\C^{2},0)$ be a sequence of  quadratic  blow-ups. The set
$\D = \pi^{-1}(0)$, the \emph{exceptional divisor} of $\pi$, is a normal crossings divisor with finitely many irreducible components --- to which
we refer simply as \emph{components} --- biholomorphic to projective lines.
We call the regular points of $\D$  \emph{trace points} and its singular points   \emph{corners}.

We denote  by  $\br_{0}$  the family of all complex analytic branches --- i.e. germs of irreducible complex analytic curves --- at $(\C^{2},0)$.
More generally, if $M$ is a complex surface, represent by
$\br_{p}$ the family   of  complex analytic branches at $p \in M$ and,
if $\CC \subset M$ is a complex analytic curve,  by $\br_{p}(\CC) \subset \br_{p}$ the   set of irreducible local components
of $\CC$ at $p$.
If $B \in \br_{0}$, we say   that the strict transform $\pi^{*}B$ is \emph{transverse} to $\D = \pi^{-1}(0)$ if $\pi^{*}B$ is smooth and touches $\D$ transversally  at a trace point.
In this case, we denote by
$D^{\pi}(B)$ the unique  component of $\D$ touched by $\pi^{*}B$.

Let $D \subset \D$ be a component of
  the exceptional divisor of the above $\pi$. We associate with $D$ an
 \emph{infinitesimal class} $\kappa(D)$, which is,
essentially,  the family of all sequences of blow-ups that ``generate $D$'':   sequences of
blow-ups   or   blow-downs over $(\tilde{M}, \D)$, the latter not  collapsing $D$, considered up to isomorphism.
For the sake of formality, we give  a more precise definition.
Let us consider, along with $(\pi,D)$, all    pairs $(\pi',D')$, where  $\pi': (\tilde{M}', \D') \to (\C^{2},0)$ is a sequence of blow-ups and
$D' \subset \D'$ is a component, such  that:
\begin{itemize}
  \item either $\pi$ factors $\pi'$, that is,    there    exists  a sequence  of blow-ups
$\varsigma$  such that $\pi' = \pi \circ \varsigma$, and $D' = \varsigma^{*}D$;
  \item or $\pi'$ factors $\pi$, that is,    there    exists  a sequence of blow-ups
$\varsigma'$  such that $\pi = \pi' \circ \varsigma'$, and $D  = \varsigma'^{*}D'$.
\end{itemize}
Next, extend this construction putting, in place of $\pi$ and $D$, any sequence of blow-ups isomorphic to $\pi$ and the isomorphic image of $D$.
The family of all these mappings is the \emph{infinitesimal class} $\kappa(D)$ of   $D \subset \D$. Any of its elements is a
 a \emph{realization} of $\kappa(D)$.
Recall that two sequences of blow-ups
$\pi^{(1)}: (\tilde{M}^{(1)} , \D^{(1)}) \to (\C^{2},0)$ and  $\pi^{(2)}: (\tilde{M}^{(2)} , \D^{(2)}) \to (\C^{2},0)$
are \emph{isomorphic} if there is   a germ of biholomorphism $\Phi: (\tilde{M}^{(1)} , \D^{(1)} ) \to (\tilde{M}^{(2)}, \D^{(2)})$ such
that $\pi^{(1)}  = \pi^{(2)} \circ \Phi$.
A realization of $\kappa(D)$ is \emph{minimal} if it is factored by no other realization of $D$. A minimal realization is unique up isomorphism.

If $\D^{*}$ is a union of  components of $\D$, then we define the infinitesimal class
$\kappa(\D^{*})$ as the family of sequences of blow-ups  that simultaneously realize  all infinitesimal classes $\kappa(D)$, for $D \subset \D^{*}$, and also separate these components.
The latter means that, if $\pi':(\tilde{M}',\D') \to (\C^{2},0)$ is a realization of $\kappa(\D^{*})$, then
 no two components of the divisor $\D'= \pi'^{-1}(0)$ corresponding to components in $\D^{*}$  intersect.
We can also define, in an evident manner, a \emph{minimal realization} for $\kappa(\D^{*})$, which is unique  up
to isomorphism.

In the same line of the   above definitions, if $\sep  \subset \br_{0}$ is a finite set,  we define its \emph{equisingularity class}  $\varepsilon(\sep)$
as the family of all sequences of blow-ups $\pi: (\tilde{M}, \D) \to (\C^{2},0)$  such
that  $\pi$ desingularizes $\sep$. By this we mean that the strict transforms $\pi^{*}B$, for $B\in \sep$, are   disjoint and
transverse to $\D$. Each such a sequence $\pi$ is a \emph{realization} of $\varepsilon(\sep)$, and we can also talk about
 \emph{minimal realizations}, all of them isomorphic.

\subsection{Dicritical structures}

Let $\pi:(\tilde{M},\D) \to (\C^{2},0)$ be a composition of  quadratic blow-ups as above. We write
$\pi = \sigma_{1} \circ \cdots \circ \sigma_{n}$ its factorization into individual blow-ups
$\sigma_{j} = (\tilde{M}_{j}, \D_{j}) \to (\tilde{M}_{j-1}, \D_{j-1})$, for $j=1,\ldots,n$, where $\D_{j} = (\sigma_{1} \circ \cdots \circ \sigma_{j})^{-1}(0)$, $\sigma_{j}$ is a punctual blow-up at $q_{j-1} \in \D_{j-1}$, with the convention that $q_{0} = 0$, $(\tilde{M}_{0},\D_{0}) = (\C^{2},0)$ and
$(\tilde{M}_{n},\D_{n}) = (\tilde{M},\D)$.
All objects and invariants pertaining to $\tilde{M}_{j}$ are said to be at \emph{level} $j$, with $j=0$ and $j=n$ being, respectively, the \emph{initial} and \emph{final levels}.
For each $j=1,\ldots,n$,  write the factorization $\pi = \pi_{j} \circ  \varsigma_{j} $ into
maps
$ \varsigma_{j}: (\tilde{M}_{n}, \D_{n}) \to (\tilde{M}_{j}, \D_{j})$
and
$\pi_{j}: (\tilde{M}_{j}, \D_{j}) \to (\tilde{M}_{0}, \D_{0})$,
where
$\varsigma_{j} = \sigma_{j+1} \circ \cdots \circ \sigma_{n}$ and
$\pi_{j} = \sigma_{1} \circ \cdots \circ \sigma_{j}$. This is depicted in the diagram:

\begin{equation}
\label{eq-sequence-blowups}
 \rlap{$\underbrace{\phantom{ (\tilde{M}_{n}, \D_{n}) \stackrel{\sigma_{n}}{\longrightarrow} \cdots  \stackrel{\sigma_{j+1}}{\longrightarrow}
 (\tilde{M}_{j}  , \D_{j})}}_{\varsigma_{j}}$}
(\tilde{M}_{n}, \D_{n}) \stackrel{\sigma_{n}}{\longrightarrow} \cdots  \stackrel{\sigma_{j+1}}{\longrightarrow}
     \overbrace{ (\tilde{M}_{j}  , \D_{j})
 \stackrel{\sigma_{j}}{\longrightarrow} (\tilde{M}_{j-1}, \D_{j-1})
 \stackrel{\sigma_{j-1}}{\longrightarrow} \cdots \stackrel{\sigma_{1}}{\longrightarrow} (\tilde{M}_{0}, \D_{0})
  }^{\pi_{j}} \simeq (\C^{2},0).
\end{equation}

\begin{ddef}
\label{def-dic-structure}
{\rm
A \emph{dicritical structure} at $(\C^{2},0)$  with underlying sequence of blow-ups
$\pi:(\tilde{M},\D) \to (\C^{2},0)$ is the set of  data $\Delta$ given, for each $j=1,\cdots,n$, by a
 decomposition $\D_{j} = \dic_{j} \cup  \ndic_{j}$, where   $\dic_{j}$ and $\ndic_{j}$
are  unions of  components of the exceptional divisor $\D_{j} = \pi_{j}^{-1}(0)$ satisfying:
\begin{enumerate}[label=(D.\arabic*)]
\item $\dic_{j}$ and $\ndic_{j}$ have no common components for $j=1,\cdots, n$;
\item $D$ is in $\dic_{j-1}$ if and only if the strict transform $\sigma_{j}^{*}D$ is in $\dic_{j}$,
for $j=2,\cdots, n$;
\item two components in $\dic_{n}$ do not intersect.
\end{enumerate}
}\end{ddef}
Components in  $\dic_{j}$ are called \emph{dicritical} or \emph{non-invariant},
whereas the ones in $\ndic_{j}$  are called \emph{non-dicritical} or \emph{invariant}.
If $D = \sigma_{j}^{-1}(q_{j-1}) \subset \dic_{j}$, we will also call the blow-up $\sigma_{j}$ dicritical, the same happening in the non-dicritical case.
We will say that $n$ is the height of $\Delta$ and we denote  $n = h(\Delta)$.
Note that, if we choose a point $p \in \D_{j}$, for $j=1,\ldots,n$, the dicritical structure $\Delta$ can be localized at $p$, by considering the decomposition into dicritical and non-dicritical components induced
by the sequence of blow-ups $\varsigma_{j}$ over $p$. We denote this localized dicritical structure by $\Delta_{p}$.
If $\varsigma_{j}$ is trivial over $p$, we say that the localized  dicritical structure is  \emph{trivial}.

Let $\Delta$ and $\Delta'$ be two dicritical structures with underlying sequences of blow-ups
$\pi = \sigma_{1} \circ \cdots \circ \sigma_{n}$ and
$\pi' = \sigma_{1}' \circ \cdots \circ \sigma_{n'}'$.
We say that $\Delta'$ \emph{dominates} $\Delta$, and we denote $\Delta' \geq \Delta$,  if:
\begin{itemize}
\item
$n' = h(\Delta') \geq  h(\Delta) = n$ and  $\sigma_{j}' = \sigma_{j}$ for $j=1,\ldots,n$;
\item  the decomponsition $\D_{j} = \dic_{j} \cup  \ndic_{j}$ is the same for $\Delta$ and $\Delta'$, for $j=1,\ldots,n$;
\item if $n'>n$, then $\varsigma_{n}' = \sigma_{n+1}' \circ \cdots \circ \sigma_{n'}'$ is a composition of non-dicritical blow-ups which is non-trivial
only over finitely many trace points of components in  $\dic_{n}$. 
\end{itemize}

\begin{example}
\label{ex-foliation-structure}
{\rm
Let $\F$ be a germ of singular holomorphic foliation at $(\C^{2},0)$, induced by a germ of holomorphic
$1$-form $\omega$ with isolated singularity at the origin. We take $\pi:(\tilde{M},\D) \to (\C^{2},0)$, a \emph{minimal reduction of singularities} for $\F$
(see \cite{seidenberg1968},
\cite{camacho1984}).
This means that we obtain a strict transform foliation $\tilde{\F} =  \pi^{*}\F$ around the
exceptional divisor $\D = \pi^{-1}(0)$ with the following properties:
\begin{itemize}
\item $\tilde{\F}$ has a finite number of
singularities over $\D$, which are all \emph{simple}, meaning that, at each of them,
  $\tilde{\F}$ is defined by a germ of holomorphic  vector field with non-nilpotent linear
part with eigenvalues with ratio outside $\Q_{+}$;
\item   $\tilde{\F}$ has no singularities and no points of tangency in the non-invariant components
of $\D$;
\item  two non-invariant components of $\D$ do not intersect;
\item  $\pi$ is minimal with these properties.
\end{itemize}
If we write the decomposition of $\pi$ as in \eqref{eq-sequence-blowups}, at each step
we have a foliation $\tilde{\F}_{j}$, the strict transform of $\F$ by $\pi_{j}$
 (with the convention that $\tilde{\F}_{n} = \tilde{\F}$), which is
a germ of singular holomorphic foliation around $\D_{j}$.
We define a  decomposition $\D_{j}= \dic_{j} \cup \ndic_{j}$  by setting
 $D \subset \dic_{j}$ if and only if $D$ is non-invariant by  $\tilde{\F}_{j}$.
 This establishes a dicritical structure that we denote by $\Delta(\F)$ or by  $\Delta(\omega)$.
}\end{example}

\subsection{Configuration of separatrices}

Considering a dicritical structure $\Delta$ at $(\C^{2},0)$, as defined in the previous subsection, we
 write the decomposition
\begin{equation}
\label{eq-connectedcomponent}
\ndic_{n} = \cup_{k=1}^{\ell} \cl{A}_{n,k}^{\iota}
\end{equation}
 into disjoint
topologically connected components.
These components, which appear in the following definition, also play an important role in some
of the arguments in this text.
For the next definition, recall that the \emph{valence} of a component $D \subset \D_{n}$, denoted $\val(D)$, is the number of points of intersection of $D$ with other components of $\D_{n}$.

\begin{ddef}
\label{def-conf-separatrices}
{\rm Let $\Delta$ be a dicritical structure at $(\C^{2},0)$.
A  \emph{configuration of separatrices} framed on
 $\Delta$
 is the collection of information $\Sigma$   given
by a finite set $\sep_{0}  \subset \br_{0}$, having a decomposition into
disjoint subsets $\sep_{0} = \sepdic_{0} \cup \sepndic_{0}$, together with its propagation along
the sequence of blow-ups  $\pi$,
$\sep_{j} = \sepdic_{j} \cup \sepndic_{j}$, where $\sepdic_{j} = \{\pi_{j}^{*}S; S \in \sepdic_{0}\}$ and
$\sepndic_{j} = \{\pi_{j}^{*}S; S \in \sepndic_{0}\}$,
 satisfying the following conditions:
\begin{enumerate}[label=(S.\arabic*)]
\item $\pi$ is a realization of the equisingularity class $\varepsilon(\sep_{0})$ and $S \in \sepdic_{0}$  if and only if  $D^{\pi}(S)  \subset \dic_{n}$;
\item for every connected component $\cl{A}_{n,k}^{\iota}$,  there is $S \in \sepndic_{0}$ such
that $D^{\pi}(S) \subset \cl{A}_{n,k}^{\iota}$;
\item if $D \subset \dic_{n}$ is such that $\val(D) = 1$, then there exists $S \in \sepdic_{0}$
such that $D^{\pi}(S) = D$;
\item $\pi$ is a minimal sequence of blow-ups that simultaneously realizes the infinitesimal class  $\kappa(\dic_{n})$ and
the equisingularity class  $\varepsilon(\sep_{0})$.
\end{enumerate}
}\end{ddef}
Taking into account Definitions \ref{def-dic-structure} and \ref{def-conf-separatrices},
we say that the pair $\DD = (\Delta,\Sigma)$ is a
 \emph{dicritical duplet} at $(\C^{2},0)$.
 Branches in $\sepdic_{j}$ and in $\sepndic_{j}$  are called, respectively,
 \emph{dicritical} and \emph{isolated separatrices}.
 The \emph{support} of $\DD$ is defined, at each level
 $j=1,\ldots,n$, as the analytic curve
$\supp_{j} = \sep_{j} \cup \ndic_{j}$, seen as germ around
$\D_{j}$.  Define, at level $j$, the singular set of $\DD$ as the singular set of its support,
$\sing(\supp_{j})$.
At the initial level, we have $\supp_{0} = \sep_{0}$.
Note that $\sepdic_{0}$
may be empty. On the other hand, taking (S.2) into account, $\sepndic_{0}$ is empty if and only if $h(\Delta) = 1$ and
 $\D_{1} =  \dic_{1}$. This is the \emph{radial case}, which is trivial for our purposes and
 will not be taken as our starting point. Nevertheless,  in some inductive constructions,
 the radial case may appear at  intermediate levels of the dicritical structure,
as a result of
the localization of our objects.
Anyhow, if we wish to include the initial radial case, we should modify condition (S.3) and also ask that, if $D \in \dic_{n}$ is such that
$\val(D) =0$ --- which implies that $n = h(\Delta) =1$ --- then there are at least two separatrices in $\sepdic_{0}$.

Given two dicritical duplets $\DD = (\Delta,\Sigma)$ and  $\DD' = (\Delta',\Sigma')$ at $(\C^{2},0)$,
we say that $\DD'$ \emph{dominates} $\DD$, and denote
$\DD' \geq \DD$,
 if:
\begin{itemize}
\item $\Delta' \geq  \Delta$;
\item $\sep_{0}'  \supset \sep_{0}$, where $\sep_{0}' $ denotes the set of separatrices of
$\Sigma'$ at level $0$, and $\sep_{0}'  \setminus \sep_{0}$  contains only isolated separatrices for $\Sigma'$;
\item separatrices in $\sep_{0}$ are qualified in the same way, as dicritical or isolated, for both $\Sigma$ and $\Sigma'$;
\item if $S \in \sep_{0}'  \setminus \sep_{0}$,
then  $D = D^{\pi}(S) \subset \dic_{n}$ and
$\pi^{*}S \cap D$ does not belong to $\supp_{n}$, the support   of $\DD$ at level $n$.
\end{itemize}
We make some comments on these conditions.
For  $S \in \sep_{0}' \setminus \sep_{0}$, let $q = \pi^{*}S \cap D$, where  $D = D^{\pi}(S)$.
Also denote by $\varsigma_{n}' = \sigma_{n+1}' \circ \cdots \circ \sigma_{n'}'$  the   blow-ups of $\pi'$ after level $n = h(\Delta)$, where
$\pi'$ is the sequence of blow-ups subjacent to $\Delta'$ and $n' = h(\Delta')$.
Since $S$ is isolated for $\Sigma'$ and $D$ is dicritical for $\Delta'$, we have that
$\varsigma_{n}'$ is non-trivial   over $q$.
We call  $q$  a \emph{escape point} and $S$ a
 \emph{escape separatrix}. 
 A detailed discussion of these objects will be carried out later in
Section \ref{section-escape}. We also refer the reader to Section 5 of \cite{canocorral2006}.

As we did for a dicritical structure $\Delta$,
we can   localize a configuration of separatrices $\Sigma$, framed on $\Delta$, at a point
$p \in \D_{j}$, for some $j=1,\ldots,n$, obtaining a configuration of separatrices $\Sigma_{p}$
framed on the dicritical structure $\Delta_{p}$.
The separatrices at level $0$ for $\Sigma_{p}$ will be precisely the branches in  $\br_{p}(\supp_{j})$.
Branches in $\br_{p}(\sep_{j})$ will
inherit  their classification as dicritical or isolated. On the other hand, the local component
of $D \subset \ndic_{j}$ at $p$ will be a dicritical separatrix of $\Sigma_{p}$ if and only if $\varsigma_{j}^{*}D$ touches
$\varsigma_{j}^{-1}(p)$ at a component of $\dic_{n}$ contained in $\varsigma_{j}^{-1}(p)$.
In particular, if $\varsigma_{j}$ is trivial over $p$, then such a local components will be an isolated separatrix.
The fact that $\Sigma_{p}$ thus defined is framed on $\Delta_{p}$ is easy to check and will
be left to the reader. We will say that $\DD_{p} = (\Delta_{p},\Sigma_{p})$ is the \emph{localization}
of $\DD = (\Delta,\Sigma)$ at the point $p$.

\begin{example}
\label{ex-foliation-separatrices}
{\rm We revisit Example \ref{ex-foliation-structure}.
As we have seen, there is a dicritical structure $\Delta(\F)$ associated with
the minimal reduction of singularities of a germ of singular holomorphic foliation  $\F$  at $(\C^{2},0)$.
Recall that a \emph{separatrix} for $\F$ is an analytic branch at $(\C^{2},0)$ invariant by $\F$.
Denote their set by $\text{Sep}_{0}(\F)$.
As proven in  \cite{camacho1982} (or in \cite{camacho1988}), $\text{Sep}_{0}(\F) \neq \emptyset$.
From this point on, we suppose that $\F$ is of \emph{generalized curve} type (see \cite{camacho1984}),
meaning that there are no saddle-node singularities --- i.e. simple singularities with one zero eigenvalue --- in the final models of its reduction of singularities.
For a generalized curve type foliation, a desingularization for $\text{Sep}_{0}(\F)$ is
a reduction of singularities for $\F$ \cite[Th. 2]{camacho1984}. A separatrix $S \in \text{Sep}_{0}(\F)$ is classified  as dicritical or isolated,
following the classification of the component $D = D^{\pi}(S)$ in the reduction divisor. Denote their sets by, respectively $\Dic_{0}(\F)$ and $\Iso_{0}(\F)$.
Isolated separatrices are finite in number. However, there are infinitely many dicritical separatrices, provided that
there is some dicritical component in $\D_{n}$. In order to define a  configuration of separatrices framed
on $\Delta(\F)$ we proceed as follows:
\begin{itemize}
\item $\sepndic_{0}(\F)= \Iso_{0}(\F)$  is the set of all isolated separatrices;
\item $\sepdic_{0}(\F)$ includes  finitely many separatrices in $\Dic_{0}(\F)$, with at least
one separatrix  attached to each component of valence one in $\dic_{n}$.
\end{itemize}
The configuration of separatrices $\Sigma(\F)$ constructed upon $\sep_{0}(\F) = \sepdic_{0}(\F) \cup \sepndic_{0}(\F)$
is  framed on $\Delta(\F)$. Indeed, (S.1) and (S.3) are true by construction, while
(S.2) follows, for instance, from \cite[Prop. 4]{mol2002}. Finally, (S.4)   is a consequence of
the minimality property of the reduction of singularities $\pi$ along with the above mentioned fact that the desingularization of the separatrices
and of the foliation are equivalent in the  generalized curve type context.
We then have a dicritical duplet by
  $\DD(\F) = (\Delta(\F),\Sigma(\F))$, also denoted by $\DD(\omega) = (\Delta(\omega),\Sigma(\omega))$ when $\F$ is defined by the germ of holomorphic $1$-form $\omega$.
This duplet is evidently non-unique when the reduction of singularities of $\F$ contains a dicritical blow-up.
}\end{example}


\section{Systems of indices and systems of residues}
\label{section-system-indices}

In  the preceding section we introduced the   object  $\DD = (\Delta,\Sigma)$, called
dicritical duplet. It joins together two nested blocks of abstract information related to singular holomorphic foliations
at $(\C^{2},0)$: a dicritical configuration  $\Delta$, which works as    a chart of   dicritical components,  and  a configuration of separatrices $\Sigma$, which is  a list of mandatory  invariant branches.
In this section,
we will introduce other two abstract blocks, systems of indices and systems of residues, that  assemble information corresponding to Camacho-Sad indices and
residues of logarithmic $1$-forms.
These four blocks   altogether will mould the construction of logarithmic foliations.

For the following definition, suppose that
 the underlying sequence of blow-ups of $\Delta$ is $\pi:(\tilde{M},\D) \to (\C^{2},0)$  and that
$n = h(\Delta)$.
Recall that the support of $\DD$ at level $j$ is
$\supp_{j} = \sep_j \cup \ndic_{j}$,
 for $j=0,\ldots,n$.
\begin{ddef}
\label{def-system-indices}
{\rm
 A \emph{system of indices} associated with the dicritical duplet  $\DD = (\Delta,\Sigma)$,
  or \emph{$\DD$-system of indices} for short, is the set of data $\Upsilon$
 obtained by
assigning, for each $j=0,\ldots,n$,    to every $p \in \supp_{j}$ and  each local branch $S \in \br_{p}(\supp_{j})$,  a number $I_{p}(S) \in \C$ in a way that  the following conditions are respected:
\begin{enumerate}[label=(I.\arabic*)]
\item Suppose that $\sigma_{j}$ is a blow-up at $q\in \D_{j-1}$,
for some $1\leq j \leq n$. Let $p \in  \supp_{j-1}$,  $S \in \br_{p}(\supp_{j-1})$, $\tilde{S}= \sigma_{j}^{*}S$ and $\tilde{p} = \D_{j} \cap \tilde{S}$. Denote by $\nu_{p}(S)$ the \emph{order} of $S$ at $p$.
Then  the following \emph{transformation law} is satisfied:
\[
 I_{\tilde{p}}(\tilde{S})  =
 \begin{cases} I_{p}(S) - \nu_{p}(S)^{2} & \text{if}\ \ p = q \\
 I_{p}(S) & \text{if}\ \ p \neq q .
 \end{cases}
\]
\item If $p$ is a regular point of $\supp_{j}$
and $S$ is the only  branch of $\supp_{j}$ at $p$, then $I_{p}(S)= 0$ (\emph{regular point condition}).
\item For $1 \leq j \leq n$  and every  component $D \subset \ndic_{j}$, the \emph{Camacho-Sad formula} is valid:
\[
\sum_{p \in D} I_{p}(D) = D \cdot D,
\]
where in the right side stands the self-intersection number of $D$ in $\tilde{M}_{j}$, which coincides with $c(D)$,
the first Chern class of the normal bundle of $D$ in $\tilde{M}_{j}$.
\item If $p \in \supp_{n}$ is a singular point and $S_{1}, S_{2}$ are the two smooth branches of
$\supp_{n}$ at $p$ (one of them necessarily in $\ndic_{n}$) then we have
the \emph{simple singularity rule}:
\[ I_{p}(S_{1}), I_{p}(S_{2}) \in \C^{*} \setminus \Q_{+}, \ \ \ \text{and} \ \ \
  I_{p}(S_{2})= 1/I_{p}(S_{1}) .
\]
\end{enumerate}
}\end{ddef}

In this situation, we say that $\TT = (\DD,\Upsilon) = (\Delta,\Sigma,\Upsilon)$ is a \emph{dicritical triplet}
at $(\C^{2},0)$.

\begin{example}
\label{ex-foliation-indices}
{\rm We resume Examples \ref{ex-foliation-structure} and \ref{ex-foliation-separatrices}.
Let $\DD(\F) = (\Delta(\F),\Sigma(\F))$ be a dicritical duplet of a germ of singular holomorphic foliation $\F$    at $(\C^{2},0)$ of generalized curve type.
We obtain a system of indices $\Upsilon(\F)$ by taking, for each
$p \in \supp_{j}$ and each $S \in \br_{p}(\supp_{j})$, the complex number
 $I_{p}(S)$ as the \emph{Camacho-Sad index} of $\tilde{\F}_{j} = \pi_{j}^{*} \F$ at $p$
(\cite{camacho1982,linsneto1986,suwa1995} and also \cite{brunella1997,suwa1998}).
 Conditions (I.1) to (I.4) are known properties of this index.
In particular, condition (I.3) follows from the theorem  asserting that the sum of Camacho-Sad indices along a compact analytic curve invariant by a foliation in
an ambient complex surface is the curve's self intersection.
We say that  $\TT(\F) = (\DD(\F),\Upsilon(\F)) = (\Delta(\F),\Sigma(\F),\Upsilon(\F))$ is a dicritical triplet associated with $\F$.
}\end{example}

A   system of indices is determined in a unique way by its prescription  at the final level,
as shown next.
\begin{prop}
Let $\DD = (\Delta,\Sigma)$ be a dicritical duplet at $(\C^{2},0)$. If $n = h(\Delta)$,  then any set of data given by complex numbers $I_{p}(S) \in \C^{*}$, for every
 $p \in \supp_{n}$ and every $S \in \br_{p}(\supp_{n})$, satisfying (I.2), (I.3) and (I.4) of Definition \ref{def-system-indices} for $j=n$,  extends in a unique way to a
 $\DD$-system of indices $\Upsilon$.
\end{prop}
\begin{proof}
The transformation law (I.1)  ought to be used in  descending   the indices assigned
at the final level and, thus, it will be automatically  fulfilled.
Besides, since (I.4) concerns only level $n$, it is enough to verify (I.2) and (I.3) for lower levels.
The proof goes by  a reverse induction argument. The initialization, at level $n$, works by hypothesis.
Suppose then that $0 \leq j< n$ and that conditions (I.1) to (I.4) are valid at all levels above $j$.

We first prove that (I.2) is true at level $j$.
Let $p$ be a smooth point of $\supp_{j}$.
If $p \in \ndic_{j}$, then $p$ is a trace point of some component $D \subset \ndic_{j}$ and there are no components of $\sep_{j}$ at $p$.
By the minimality condition  (S.4), $\varsigma_{j}$ would be non-trivial at $p$ if and
only if $\varsigma_{j}^{-1}(p)$ contained some dicritical component, say $D' \subset \dic_{n}$. But
we would have that either $D'$ disconnects $\ndic_{n}$ or  $\val(D') = 1$. In both cases, by (S.2) and (S.3),
we would find some element of $\sep_{j}$ at $p$, which is impossible.
Hence, $\varsigma_{j}$ is a local isomorphism over $p$. Thus, using (I.2) at level $n$ and  successively
applying  (I.1), we conclude that $I_{p}(D) = 0$.

Suppose now that $p \not\in \ndic_{j}$. Then, $p$ is a trace point of a component $D \subset \dic_{j}$
and there is a unique $S \in \sep_{j}$ at $p$, which is smooth. Suppose first that $S \in \sepdic_{j}$. If $\varsigma_{j}$ were non-trivial over $p$, then, by (D.3),  $\varsigma_{j}^{-1}(p)$ would contain some non-dicritical component, which in turn would imply the existence of a least one element of $\sepndic_{j}$ at $p$, which is not allowed.
Thus $\varsigma_{j}$ is trivial over $p$ and, as in the previous paragraph,
we find  that  $I_{p}(S) = 0$.
On the other hand, if $S \in \sepndic_{j}$, then $S$ is tangent to $D$ at $p$, $\varsigma_{j}$ is non-trivial over $p$ and,
as a consequence of (S.2), $\varsigma_{j}^{-1}(p)$   contains only non-dicritical components.
By  (S.4),  over $p$,
  $\varsigma_{j}$ is the minimal sequence of blow-ups that separates $D$ and $S$. Thus,
$\varsigma_{j}^{-1}(p)$ is a linear chain of components $D_{1} \cup \cdots \cup D_{r} \subset \ndic_{n}$,
with $\varsigma_{j}^{*}D \cap D_{1} \neq \emptyset$
and $D_{i} \cap D_{i+1} \neq \emptyset$ for $i=1,\ldots, r-1$. Also,
$\tilde{S} = \varsigma_{j}^{*}S$ touches $D_{1}$ at a point $\tilde{p}$ and no other components
of $\sepndic_{n}$ touches $\varsigma_{j}^{-1}(p)$. We have $D_{1} \cdot D_{1} = -1$ and
$D_{i} \cdot D_{i} = -2$ for $i=2,\ldots,r$. By the application, at level $n$, of (I.3) and (I.4)  along this linear chain,
we find that $I_{\tilde{p}}(\tilde{S}) = -r$. Finally, the successive application of (I.1) gives
that $I_{p}(S)= 0$. This proves (I.2) at level $j$.

In order to prove (I.3) at level $j$,
let $D \subset \ndic_{j}$ be an irreducible component.
Suppose that $\sigma_{j+1}$ is a blow-up at $p \in \D_{j}$.
If $p \not\in D$,
 the Camacho-Sad formula (I.3) will evidently hold for $D$, since it holds for
$\tilde{D} = \sigma_{j+1}^{*}D$ and
 all invariants involved are preserved by  $\sigma_{j+1}$.
Now, in the case $p \in D$, we have, on the one hand,
$\tilde{D} \cdot \tilde{D} = D \cdot D - 1$.
On the other hand,
if $\tilde{p} = \tilde{D} \cap \sigma_{j+1}^{-1}(p)$, then the transformation law (I.1) gives
$I_{\tilde{p}}(\tilde{D}) = I_{p}(D)-1$. The Camacho-Sad formula then holds for $D$, since the
indices of $D$ at points other than $p$ are preserved by $\sigma_{j+1}$. This proves (I.3) at
level $j$, concluding the proof of the proposition.
\end{proof}

Recall that a meromorphic $1$-form $\eta$
--- and also the germ of singular holomorphic foliation $\LL$ that it
defines ---
  is \emph{logarithmic}   if it can be written as
\begin{equation}
\label{eq-log-form}
\eta = \sum_{S \in \sep} \lambda_{S} \frac{d f_{S}}{f_{S}} + \alpha ,
\end{equation}
where $\alpha$ is some closed germ of holomorphic $1$-form,
 $\sep \subset \br_{0}$ is a finite set and, for every $S \in \sep$,
 $f_{S} \in \mathcal{O}_{0}$ is  a local irreducible  equation  for $S$
 and  $\lambda_{S} \in \C^{*}$.
 The branches in $\sep$ are components of the polar set  of the closed meromorphic
 $1$-form $\eta$ and, hence, are
 are separatrices of the foliation $\LL$.

We define a dicritical duplet $\DD(\eta) = (\Delta(\eta),\sigma(\eta))$ intrinsically  associated with the
logarithmic $1$-form $\eta$. The dicritical structure is
$\Delta(\eta) = \Delta(\LL)$, the one defined by the reduction of singularities of the foliation $\LL$, as in Example \ref{ex-foliation-structure}.
As for the configuration of separatrices,
in $\Sigma(\eta)$ we include all branches in $\sep$ and also all isolated separatrices of $\LL$.
For $\Sigma(\eta)$  to be a configuration of separatrices, it remains to assure that  condition (S.3) is true, in other words, that
there is $S \in \sep$ such that $D^{\pi}(S) = D$ for every
component $D \subset \dic_{n}$ with $\val(D) = 1$, where $n = h(\Delta(\eta))$ and
$\pi$ is the sequence of blow-ups beneath  $\Delta(\eta)$.
Nevertheless, this   happens in the cases considered in this text, more specifically, if we assume that $\eta$ is a faithful logarithmic $1$-form (see Definition \ref{def-faithful} and Section \ref{section-escape}).
Thus,   henceforth we can consider that $\DD(\eta) = (\Delta(\eta),\Sigma(\eta))$ is a dicritical duplet.
We obtain a dicritical triplet $\TT(\eta) = (\DD(\eta), \Upsilon(\eta)) = (\Delta(\eta),\sigma(\eta),\Upsilon(\eta))$ by prescribing as indices those provided
by the Camacho-Sad indices of $\LL$.


In an attempt to handle residues of logarithmic $1$-forms,  we add another    element to our construction:

\begin{ddef}
\label{def-system-residues}
{\rm Let $\DD = (\Delta,\Sigma)$ be a dicritical duplet at $(\C^{2},0)$.
A \emph{system of residues} for $\DD = (\Delta,\Sigma)$, or simply
a \emph{$\DD$-system of residues},   is the set of data $\Lambda$
obtained by assigning
numbers $\lambda_{S} \in \C^{*}$  to each irreducible component $S \subset \supp_{j}$ ($S$ is either a
component of $\ndic_{j}$ or a germ of separatrix  of $\sep_{j}$), for $j=0,\ldots,n$,
  obeying  the following rules:
\begin{enumerate}[label=(R.\arabic*)]
\item
Let $S \subset \supp_{j}$ and $\sigma_{j+1}$ be the factor of $\pi$ corresponding to a blow-up at a point
$q \in \supp_{j}$, for some $j<n$. If $\tilde{S} = \sigma_{j+1}^{*}S$, then $\lambda_{\tilde{S}} = \lambda_{S}$
(\emph{invariance under blow-ups}).
\item  If $\sigma_{j+1}$ is a non-dicritical blow-up at $q \in \supp_{j}$, for some $0 \leq j < n$,
and $D = \sigma_{j+1}^{-1}(q)$, then
\[\lambda_{D} = \sum_{S \subset \supp_{j}} \nu_{q}(S) \lambda_{S} \neq 0 .\]
\item  If $\sigma_{j+1}$ is a dicritical blow-up at $q \in \supp_{j}$, for some $0 \leq j < n$, then
\[\sum_{S \subset \supp_{j}} \nu_{q}(S) \lambda_{S} = 0 \ \ \ \text{(\emph{main resonance})}.\]
\end{enumerate}
}\end{ddef}
The sum in items (R.2) and (R.3) has finitely many non-zero terms, only those corresponding to components $S$
containing $q$.
We call it \emph{weighted balance of residues}.
It is  worth pointing out that the assignment of residues at level $0$
determines, by the successive application of rules (R.1) and (R.2),
residues for all components of the support   at all levels. Evidently, rules (R.2) and (R.3)
impose  compatibility conditions on  a set of residues at level 0 that actually engender a $\DD$-system of residues.
Two $\DD$-systems of residues $\Lambda$ and $\Lambda'$
are \emph{equivalent} if  there exists a common ratio $\alpha \in \C^{*}$
for corresponding residues. We denote, in this case, $\Lambda = \alpha \Lambda'$.
If $\Lambda$ is a $\DD$-system of residues, then the principal part of the logarithmic $1$-form shown  in equation \eqref{eq-log-form}
for $\sep = \sep_{0}$ is denoted by $\eta_{\Lambda}$. If $\Lambda$ and $\Lambda'$ are equivalent $\DD$-systems of residues,
with a proportionality ratio $\alpha \in \C^{*}$ as above, then  $\eta_{\Lambda} = \alpha \, \eta_{\Lambda'}$, the two $1$-forms   defining
the same germ of singular holomorphic foliation at $(\C^{2},0)$.

We establish a bond between
 systems of indices and
 systems of residues with the following:
\begin{ddef}
\label{def-consistent}
{\rm  Let $\Upsilon$   and $\Lambda$ be
$\DD$-systems of indices and residues,
where
$\DD = (\Delta,\Sigma)$ is a dicritical duplet at $(\C^{2},0)$ of height  $n = h(\Delta)$.
 We say that $\Upsilon$ and $\Lambda$ are \emph{consistent}
if, for every   $p \in \sing(\supp_{n})$,
\[ I_{p}(S_{1}) = - \frac{\lambda_{S_{2}}}{\lambda_{S_{1}}} \qquad \text{and} \qquad
I_{p}(S_{2}) = - \frac{\lambda_{S_{1}}}{\lambda_{S_{2}}} ,\]
where $S_{1}, S_{2}$ are the two local components
of $\br_{p}(\supp_{n})$.
In this situation, we say that $\QQ = (\Delta,\Sigma,\Upsilon,\Lambda)$ forms a \emph{dicritical quadruplet}.
}\end{ddef}
Note that, by (I.4), the two equalities in the definition are equivalent. Note also that there are no
conditions on the   separatrices in
$\sepdic_{n}$, and this gives some degree of freedom in their assignment of residues.
This remark will be particularly useful in Section \ref{section-escape}.
Another important point   is that, if $\Lambda$ and $\Lambda'$ are equivalent
$\DD$-systems of residues, then $\Lambda$ is consistent with a $\DD$-system of indices $\Upsilon$ if and only if  $\Lambda'$ is.
It is apparent that, given $p \in \D_{j}$, both a system of indices $\Upsilon$ and a system of
residues $\Lambda$ can be localized at $p$, keeping their consistency relation, if it is the case.
Thus, denoting these localizations by $\Upsilon_{p}$ and  $\Lambda_{p}$, we can
consider the localizations $\TT_{p}$ and $\QQ_{p}$ at $p$ of, respectively, a dicritical
triplet $\TT = (\Delta,\Sigma, \Upsilon)$ and
a dicritical
quadruplet $\QQ = (\Delta,\Sigma, \Upsilon,\Lambda)$.

It is always possible to assign residues at  the final level
in a consistent way with a  system of indices. More precisely:
\begin{prop}
\label{prop-residues-leveln}
Suppose that $\TT = (\Delta,\Sigma,\Upsilon)$ is a dicritical triplet at $(\C^{2},0)$. Then,  it is possible to assign
numbers $\lambda_{S}  \in \C^{*}$ for the irreducible components $S \subset \supp_{n}$,
where $n = h(\Delta)$, in such a way that  the consistency condition of    Definition
\ref{def-consistent} is satisfied.
\end{prop}
\begin{proof}
We can work separately in each
topologically connected component   defined in
\eqref{eq-connectedcomponent} and we fix some  $\A=\A_{n,k}^{\iota}$.
Our task is
to assign residues to each component of $\supp_{n}$ intersecting $\A$.
To that end, after fixing an arbitrary  component
$D  \subset \A$, define a partial order in
the   components of $\A$  in the following inductive way:
$D$ is the maximal element, its immediate predecessors are the components $D'$  such that
$D \cap D' \neq \emptyset$ and so forth.
This is possible since the dual graph of $\A$ is a tree
(see the paragraph preceding Proposition \ref{prop-rational-index}).
Choose  an arbitrary value $\lambda_{D} \in \C^{*}$.
Then, by the  consistency condition in Definition \ref{def-consistent},
$\lambda_{D}$  determines all residues $\lambda_{S}$ for components $S \subset \supp_{n}$ intersecting
$D$: we must set
$\lambda_{S} = -\lambda_{D} I_{p}(D)$, where $p = S \cap D$.  In particular, we find all residues $\lambda_{D'}$
for  components $D'$ preceding $D$.  Again, for each such $D'$,
all residues of components of $\supp_{n}$ intersecting $D'$ are calculated by the consistency condition.
This clearly provides an inductive process that will end once we reach the minimal elements with respect to this order.
\end{proof}

Proposition \ref{prop-residues-leveln} will provide the initial step, in a reverse induction process,
of the construction of  a $\DD$-system of residues $\Lambda$ consistent with a $\DD$-system of indices $\Upsilon$.
The general step of the induction is the  \emph{amalgamation procedure},    described  in the next section.
For it,
we   need
 the proposition below.
All elements of its proof are in \cite[Prop. 3]{corral2012} and \cite[Lem. 15]{canocorral2006}.
We include a proof here adapted to our setting.

\begin{prop}
\label{prop-residues-indicies}
Let $\QQ= (\Delta,\Sigma,\Upsilon,\Lambda)$ be a dicritical quadruplet at $(\C^{2},0)$.
If $S \in \br_{0}(\sep_{0})$,
then
\[ \lambda_{S} I_{0}(S) = - \sum_{S' \in \br_{0}(\sep_{0}) \setminus S} \lambda_{S'} (S',S)_{0},\]
where $( \, \cdot \, ,  \cdot \, )_{0}$ denotes the
intersection number at $0 \in \C^{2}$.
\end{prop}
\begin{proof}
We will use the following induction assertion at level $j$, for $0 \leq j \leq n$:
if $p \in \supp_{j}$ and $S$ is a local component of $\supp_{j}$ at $p$, then
\[ \lambda_{S} I_{p}(S) = - \sum_{S' \in \br_{p}(\supp_{j}) \setminus S} \lambda_{S'} (S',S)_{p}.\]
We are evidently employing here data from the localized quadruplet $\QQ_{p} = (\Delta_{p},\Sigma_{p},\Upsilon_{p},\Lambda_{p})$.

At level $n$, the  assertion is true: it is trivial for each regular point of $\supp_{n}$ and, for singular points, it is the very
definition of consistency. Let $1 \leq j \leq n$ and
suppose that the result is true, as stated, for all points in $\supp_{j}$.
We will prove that it is also true for points at level $j-1$. Suppose that $p \in \supp_{j-1}$. If the blow-up $\sigma_{j}$
is trivial over $p$, there is nothing to prove and,
thus, we are reduced to the case where $\sigma_{j}$ is a blow-up at $p$.

Let us denote $\tilde{S}$ and
$\tilde{S'}$ the strict transforms by $\sigma_{j}$ of the components $S$ and $S'$,  $D = \sigma_{j}^{-1}(p)$
and
$\tilde{p}= \tilde{S} \cap D$. We have

\begin{eqnarray*}
\sum_{S' \in \br_{p}(\supp_{j-1}) \setminus S} \lambda_{S'} (S',S)_{p}  & =  & \sum_{S' \in \br_{p}(\supp_{j-1}) \setminus S} \lambda_{S'}  \left(\nu_{p}(S') \nu_{p}(S)
 + (\tilde{S}',\tilde{S})_{\tilde{p}} \right)  \\
 & = & \nu_{p}(S) \sum_{S' \in \br_{p}(\supp_{j-1}) \setminus S} \lambda_{S'}   \nu_{p}(S')  +
\sum_{S' \in \br_{p}(\supp_{j-1}) \setminus S} \lambda_{S'} (\tilde{S}',\tilde{S})_{\tilde{p}} \\
& = & -\lambda_{S} \nu_{p}(S)^{2} +
\nu_{p}(S) \underbrace{ \sum_{S' \in \br_{p}(\supp_{j-1})} \lambda_{S'}   \nu_{p}(S')}_{\displaystyle(*)}  +
\underbrace{\sum_{S' \in \br_{p}(\supp_{j-1}) \setminus S} \lambda_{\tilde{S}'} (\tilde{S}',\tilde{S})_{\tilde{p}}}_{\displaystyle(**)}
\end{eqnarray*}

If $\sigma_{j}$ is a dicritical blow-up, then the summation $(*)$ vanishes by  (R.3),
whereas $(**)$  can be calculated using the
 the induction assertion at level $j$ for $\tilde{p}$,
yielding,  by finally applying (R.1) and (I.1),
\[
\sum_{S' \in \br_{p}(\supp_{j-1}) \setminus S} \lambda_{S'} (S',S)_{p}   =
  -\lambda_{S} \nu_{p}(S)^{2}  - \lambda_{\tilde{S}} I_{\tilde{p}}(\tilde{S})
= -\lambda_{S}( \nu_{p}(S)^{2} + I_{\tilde{p}}(\tilde{S}) ) = -\lambda_{S} I_{p}(S),
\]
proving thereby the result in this case.

On the other hand, if $\sigma_{j}$ is a non-dicritical blow-up, the summation $(*)$ equals
 $\lambda_{D}$  by (R.2) and, by  applying the induction assertion for $\tilde{p}$
 followed by (R.1) and (I.1), we get
 \begin{eqnarray*}
\sum_{S' \in \br_{p}(\supp_{j-1}) \setminus S} \lambda_{S'} (S',S)_{p}  & =  &
 -\lambda_{S} \nu_{p}(S)^{2} +
\lambda_{D}(D,\tilde{S})_{\tilde{p}}  +
\sum_{S' \in \br_{p}(\supp_{j-1}) \setminus S} \lambda_{\tilde{S}'} (\tilde{S}',\tilde{S})_{\tilde{p}} \\
& = &
 -\lambda_{S} \nu_{p}(S)^{2} - \lambda_{\tilde{S}} I_{\tilde{p}}(\tilde{S})
 = -\lambda_{S}( \nu_{p}(S)^{2} + I_{\tilde{p}}(\tilde{S}) ) = -\lambda_{S} I_{p}(S).
\end{eqnarray*}
This concludes the proof of the proposition.
\end{proof}

\section{Existence of dicritical quadruplets and logarithmic $1$-forms}
\label{section-existence-quadruplet}

We start by recalling a   definition from \cite{canocorral2006} that will be crucial to the development of the results of this section. Let $\eta$ be a germ of logarithmic $1$-form at $(\C^{2},0)$,
 written as
\begin{equation}
\label{eq-log-form1}
 \eta = \lambda_{1} \frac{d f_{1}}{f_{1}} + \cdots +  \lambda_{r} \frac{d f_{r}}{f_{r}} + \alpha,
\end{equation}
where $\alpha$ is a germ of closed holomorphic $1$-form,  $r \geq 2$, and, for $i=1,\ldots,r$,  $f_{i} \in \cl{O}_{0}$ is irreducible with order $\nu_{i} = \nu_{0}(f_{i})$
and
$\lambda_{i} \in \C^{*}$. Denote by $\LL$ the singular holomorphic foliation defined by $\eta$.
\begin{ddef}
\label{def-faithful}
{\rm The $1$-form
$\eta$ is \emph{$1$-faithful} if
$$\nu_{0}( f\eta) = m-1,$$
where $f = f_{1} \cdots f_{r}$ and $m = \nu_{0}(f) = \nu_{1}+ \cdots + \nu_{r}$.
We simply say  that $\eta$ is  \emph{faithful} if its strict transforms are $1$-faithful  along the reduction of singularities of $\LL$.
}\end{ddef}
In the definition, the closed holomorphic $1$-form $\alpha$ in \eqref{eq-log-form1} is inert, since
$\nu_{0}( f\alpha) \geq m$. Also, the definition does not depend on the choice of irreducible equations
$f_{1}, \ldots, f_{r}$ for the components of the polar set of $\eta$. This gives, in particular, the following remark, that will
be used a couple of times in the proof of Theorem \ref{teo-quadruplet} below:
if $\eta$ and $\eta'$ are   germs of logarithmic $1$-forms at $(\C^{2},0)$ differing from each other
by a germ of closed holomorphic $1$-form, then  $\eta$ is faithful
if and only if $\eta'$ is. In Section \ref{section-escape} we make a brief discussion on the
notion of faithfulness. Notably,  the following result, from \cite[Lem. 7]{canocorral2006}, will be
explained: if
 $\lambda_{0} = \sum_{i=1}^{r} \nu_{i}  \lambda_{i}$ is the weighted balance of residues of $\eta$,
then
\begin{itemize}
  \item   $\lambda_{0} \neq 0$ implies that $\eta$ is non-dicritical at $0 \in \C^{2}$ and $1$-faithful;
  \item    $\lambda_{0} = 0$ and $\eta$ $1$-faithful implies that $\eta$ is  dicritical $0 \in \C^{2}$ (\emph{main resonant} case).
\end{itemize}
In other words, for a faithful logarithmic $1$-form, the dicritical or non-dicritical character of
a blow-up in its reduction of singularities is determined by,  respectively, the vanishing or non-vanishing
of the weighted balance of residues.
Another important point is the following: if $\eta$ is a faithful logarithmic $1$-form
and $\pi :(\tilde{M} ,\D) \to (\C^{2},0)$ is
the composition of blow-ups of its  reduction
of singularities  up to  the appearance of escape points,
then  all  non-dicritical   components of $\D$ are poles of $\pi^{*} \eta$, while the dicritical components are neither poles nor contained in the
its set of zeroes.  This can be seen below, in Section \ref{section-escape}, from formula  \eqref{eq-sigma*eta}  and by the discussion following formula  \eqref{eq-tilde-omega}.
Then, the existence of at least one separatrix attached to each dicritical component $D \subset \D$ of valence one, a necessary condition in order that $\DD(\eta)$  be a dicritical duplet, follows from the Residue Theorem applied to the restriction of $\pi^{*} \eta$ to $D$.
Bearing in mind these observations, we make the following  definition:

\begin{ddef}
\label{def-Q-logarithmic}
{\rm  Let  $\QQ =(\Delta,\Sigma,\Upsilon,\Lambda)$ be a dicritical quadruplet at $(\C^{2},0)$.
 We say that a germ of  closed meromorphic $1$-form $\eta$ is \emph{$\QQ$-logarithmic} if:
 \begin{enumerate}
  \item there is a writing
   \[ \eta = \sum_{S \in \sep_{0}} \lambda_{S} \frac{d f_{S}}{f_{S}} + \alpha ,\]
where $\sep_{0}$ is the level $0$ of $\Sigma$, $f_{S} \in \mathcal{O}_{0}$  are local  equations for the components
 $S \in \sep_{0}$,   $\lambda_{S}$ are the residues assigned by $\Lambda$ and
$\alpha$ is a germ of closed holomorphic $1$-form;
\item $\DD(\eta) \geq \DD$, where $\DD(\eta)  = (\Delta(\eta),\Sigma(\eta))$ and $\DD = (\Delta,\Sigma)$;
\item $\eta$ is faithful.
\end{enumerate}
We say that $\eta$ is \emph{strictly $\QQ$-logarithmic} if $\DD(\eta) = \DD$.
}\end{ddef}
Above,  we also say  that $\eta$ is a \emph{logarithmic model} for $\QQ$.
Definition \ref{def-Q-logarithmic} deserves a few comments.
First, the poles of $\eta$ are prescribed by $\sep_{0}$, the level $0$ of the configuration of separatrices $\Sigma$, and their
  residues are those defined by the system of residues $\Lambda$.
Second, the dominance condition $\DD(\eta) \geq \DD$ says, aside from the fact that the reduction of singularities
of the foliation $\LL$ is longer than the underlying sequence of blow-ups of $\Delta$, that the isolated separatrices of $\LL$ that
are not poles of $\eta$ are escape separatrices with respect to $\DD =(\Delta,\Sigma)$.
Finally, the axiomatics defining systems of residues and systems of indices, along with the notion of consistency in
Definition \ref{def-consistent}, gives that the Camacho-Sad indices of $\LL$ with respect to the separatrices in $\sep_{0}$
are precisely those given by  $\Upsilon$.
We should draw attention to the fact that asking $\eta$ to be faithful is stronger than asking that common dicritical and non-dicritical components of $\Delta(\eta)$ and $\Delta$
correspond, which is part of the information given by  condition (2).
Actually, we are demanding that whenever the weighted balance of residues $\lambda_{0}$ vanishes, the corresponding blow-up is dicritical, but,
 additionally, that it is dicritical in a ``1-faithful'' way.

The following theorem is a cornerstone
to all subsequent results of this article.
The remainder of this section is devoted to its proof.
\begin{maintheorem}
\label{teo-quadruplet}
Let  $\TT =(\Delta,\Sigma,\Upsilon)$ be a dicritical triplet. Then there exists:
\begin{enumerate}
  \item a system of
residues $\Lambda$ such that $\QQ = (\TT,\Lambda) = (\Delta,\Sigma,\Upsilon,\Lambda)$ is a dicritical quadruplet;
  \item  a $\QQ$-logarithmic $1$-form $\eta$.
\end{enumerate}
\end{maintheorem}
\begin{proof}
The two assertions above are proved simultaneously. We essentially adapt the arguments in \cite{canocorral2006}
to the combinatorics involved in the definition of a dicritical quadruplet.
The idea is to descend the dicritical structure, starting at the final level $n = h(\Delta)$.
For a fixed level $j$, we will furnish the following objects:
\begin{enumerate}[label=(\roman*)]
\item a \emph{partial system of residues}  $\Lambda^{(j)}$, that is to say,
 residue data for all components of the support for levels $j$ and higher,
 satisfying (R.1), (R.2) and (R.3),
consistent with the  indices of $\Upsilon$ at level $n$.
\item  for each   $p \in \supp_{j}$, a $\QQ^{(j)}_{p}$-logarithmic $1$-form $\eta^{(j)}_{p}$, where
$\QQ^{(j)}_{p} = (\TT_{p},\Lambda^{(j)}_{p})$ is the dicritical quadruplet at $(\tilde{M}_{j},p)$ defined by the localizations
of $\TT$ and of  $\Lambda^{(j)}$ at $p$.
\end{enumerate}

At  level $n$, residues of $\Lambda^{(n)}$  are provided by Proposition  \ref{prop-residues-leveln}.
As for the $\QQ^{(n)}_{p}$-logarithmic $1$-forms,
at a singular point $p \in \supp_{n}$, we take local coordinates $(x,y)$ at $p$ such that
$x=0$ and $y=0$ are equations of the two components of $\supp_{n}$ at $p$, having residues $\lambda_{1}$ and $\lambda_{2}$, and
then we   set $\eta^{(n)}_{p} = \lambda_{1} dx/x + \lambda_{2} dy/y$.
At a regular point $p \in \supp_{n}$, in local coordinates $(x,y)$ at $p$ such that
 the   component  of $\supp_{n}$ at $p$ has equation $x=0$, we set $\eta^{(n)}_{p} = \lambda dx/x$,
 where $\lambda \in \C^{*}$ is the corresponding residue.

Suppose that, for a fixed $j$ with $1 \leq j \leq n$,  residues and logarithmic 1-forms  have been provided as described
in   items (i) and (ii) above. We will amalgamate data from level $j$ in order to construct, at level $j-1$, a
 partial system of residues $\Lambda^{(j-1)}$, as well as a $\QQ^{(j-1)}_{p}$-logarithmic $1$-form
for every $p \in \supp_{j-1}$, where  $\QQ^{(j-1)}_{p} = (\TT_{p},\Lambda^{(j-1)}_{p})$.
Actually, we only need to do this for the point
  $p \in \supp_{j-1}$ that is the center of the blow-up $\sigma_{j}$.
Denote by $D = \sigma_{j}^{-1}(p)$ and let $p_{1},\ldots,p_{\ell}$ be
 the points of intersection of  $D$ with $\overline{\supp_{j} \setminus D}$.
 In our inductive construction, in addition to the partial system of residues $\Lambda^{(j)}$,
 for each $k=1,\ldots,\ell$,   we have
a germ of $\QQ^{(j)}_{p_{k}}$-logarithmic
$1$-form $\eta_{p_{k}}^{(j)}$ at $p_{k}$, where
$\QQ^{(j)}_{p_{k}} = (\TT_{p_{k}},\Lambda^{(j)}_{p_{k}})$.

The non-dicritical and dicritical cases are then treated separately.

 \smallskip \smallskip
\par \noindent
\underline{Case 1}: $\sigma_{j}$ is a non-dicritical blow-up.
The amalgamation of data from level $j$ is somewhat instantaneous here and will be a consequence
of Proposition \ref{prop-residues-indicies}.
Indeed, for each $p_{k} \in D$, we have
\begin{equation}
\label{eq-lambdaD}
 \lambda_{D} I_{p_{k}}(D) = - \sum_{\tilde{S} \in \br_{p_{k}}(\supp_{j}) \setminus D} \lambda_{\tilde{S}} (\tilde{S},D)_{p_{k}},
 \end{equation}
where
residues are those of $\Lambda^{(j)}$.
Applying  the Camacho-Sad formula (I.3) followed by equation \eqref{eq-lambdaD}, we get
\begin{eqnarray}
\label{eq-amalgamation1}
\lambda_{D} &=&   \lambda_{D}\left(- \sum_{k=1}^{\ell}  I_{p_{k}}(D)\right) \ = \   \sum_{k=1}^{\ell}(- \lambda_{D} I_{p_{k}}(D))   \\
\nonumber    &=&
\sum_{k=1}^{\ell}  \left( \sum_{\tilde{S} \in \br_{p_{k}}(\supp_{j}) \setminus D}  \lambda_{\tilde{S}} (\tilde{S},D)_{p_{k}} \right) \\
\nonumber    &=& \sum_{S \in \br_{p}(\supp_{j-1})} \nu_{p}(S) \lambda_{\tilde{S}},
\end{eqnarray}
where $\tilde{S} = \sigma_{j}^{*}S$ is the strict transform  of $S \in \br_{p}(\supp_{j-1})$, which
belongs to some $\br_{p_{k}}(\supp_{j})$.
Equation \eqref{eq-amalgamation1} allows us to
define $\lambda_{S} = \lambda_{\tilde{S}}$, in such a way that condition (R.2) is satisfied by the blow-up
$\sigma_{j}$.
Therefore, we can define residues at level $j-1$ by simply importing then from level $j$:
for every component $S \subset \supp_{j-1}$, set $\lambda_{S} = \lambda_{\tilde{S}}$, where $\tilde{S} = \sigma_{j}^{*}S$.
In this way, we have a partial system of residues $\Lambda^{(j-1)}$.

In our downwards construction, write the
$\QQ^{(j)}_{p_{k}}$-logarithmic
$1$-form  at $p_{k} \in D = \sigma_{j}^{-1}(p)$ as
\[ \eta_{p_{k}}^{(j)} = \lambda_{D} \frac{d f_{D}}{f_{D}} +
\sum_{\tilde{S} \in \br_{p_{k}}(\supp_{j}) \setminus D} \lambda_{\tilde{S}}  \frac{d f_{\tilde{S}}}{ f_{\tilde{S}}} + \alpha_{p_{k}},\]
where
$f_{\tilde{S}} \in \cl{O}_{p_{k}}$ are local equations for $\tilde{S} \in \br_{p_{k}}(\supp_{j}) \setminus D$,
$f_{D} \in \cl{O}_{p_{k}}$ is a local equation for $D$ and $\alpha_{p_{k}}$ is a germ of closed holomorphic $1$-form at $p_{k}$.
For each $S \in \br_{p}(\supp_{j-1})$, take $f_{S} \in \cl{O}_{p}$  an equation of $S$ such
that $f_{S} \circ \sigma_{j} =  f_{D}^{\nu_{p}(S)} f_{\tilde{S}}$, where $\tilde{S} = \sigma_{j}^{*}S$
and $f_{D}$ is the local equation of $D$ at $p_{k} = \tilde{S} \cap D$.
We then define
\[ \eta_{p}^{(j-1)} =
\sum_{S \in \br_{p}(\supp_{j-1})} \lambda_{S} \frac{d f_{S}}{f_{S}}.\]
Setting $\QQ_{p}^{(j-1)} = (\TT_{p}, \Lambda^{(j-1)}_{p})$, then  $\eta_{p}^{(j-1)}$ is $\QQ_{p}^{(j-1)}$-logarithmic.
Indeed, by equation \eqref{eq-amalgamation1},
the residue  of  $ \sigma_{j}^{*} \eta_{p}^{(j-1)} $ with respect to
 $D =   \sigma_{j}^{-1}(p)$
 is
precisely $\lambda_{D} \neq 0$, which  simultaneously implies that $\eta_{p}^{(j-1)}$ is $1$-faithful and
that $D$ is invariant by $ \sigma_{j}^{*} \eta_{p}^{(j-1)}$ \cite[Lem. 7]{canocorral2006}. On the other hand, faithfulness of $\eta_{p}^{(j-1)}$
at  higher levels  is assured by  the inductive construction, since, at each point $p_{k} \in D$,
the logarithmic $1$-forms $\sigma_{j}^{*} \eta_{p}^{(j-1)}$ and $\eta_{p_{k}}^{(j)}$
differ from each other by a germ of
 closed holomorphic $1$-form.

 \smallskip \smallskip
\par \noindent
\underline{Case 2}: $\sigma_{j}$ is a dicritical blow-up.
We  consider again $p_{1},\ldots,p_{\ell}$, which now are precisely
 the  points of $\supp_{j}$ on $D = \sigma_{j}^{-1}(p)$. We start by proving
 the following:
\begin{assertion} $\ell \geq 2$.
\end{assertion}
\begin{proof} Indeed, if $j=n$, so that $D \subset \dic_{n}$, the result is clear:
recalling that, by condition (D.3), $\val(D)$ is actually the number of points of intersection
of $D$ with $\ndic_{n}$, the result
  is obvious if $\val(D) \geq 2$ and, if $\val(D) = 1$, it is a consequence of condition (S.3).
Suppose now that   $1 \leq j < n$ and $D \subset \dic_{j}$.
If $\varsigma_{j}$ is trivial over $D$, then the result follows from the case $j=n$.
When this is not the case, we
remark that, for each point $q  \in D$    over which $\varsigma_{j}$ is non-trivial, we find,
as a consequence of (D.3) and (S.2),
at least one separatrix in $\sepndic_{j}$ at $q$.
Besides, each component of  $ \D_{j}$, other than $D$, intersecting
$D$ at a point $q'$ over which $\varsigma_{j}$ is trivial is necessarily in $\ndic_{j}$ by (D.3).
We are then reduced to the situation where there is only one such  point $q  \in D$ and none of those points
$q' \in D$. But, in this case, denoting $\tilde{D} = \varsigma_{j}^{*}D$, then $\val(\tilde{D}) = 1$   in $\D_{n}$. By (S.3), we must have
a separatrix of $\sepdic_{j}$ touching $D$ at a point other than $q$, proving thereby the assertion.
\end{proof}

Now, at $p_{k} \in D = \sigma_{j}^{-1}(p)$, the given $\QQ_{p_{k}}^{(j)}$-logarithmic $1$-form
is written as
\[ \eta_{p_{k}}^{(j)} =
\sum_{\tilde{S} \in \br_{p_{k}}(\supp_{j})  }  \lambda_{\tilde{S}}  \frac{d f_{\tilde{S}}}{ f_{\tilde{S}}} + \alpha_{p_{k}},\]
where $f_{\tilde{S}} \in \cl{O}_{p_{k}}$ is a local equation for $\tilde{S}$
 and $\alpha_{p_{k}}$ is a germ of closed holomorphic $1$-form at $p_{k}$.

For each $k=1,\ldots,\ell $,
denote by $\br_{p}^{p_{k}}(\supp_{j-1}) \subset \br_{p}(\supp_{j-1}) $ the subset of all branches $S$ such that $\tilde{S}= \sigma_{j}^{*}S \in \br_{p_{k}}(\supp_{j})$.
Also define
\[\rho_{p_{k}} = \sum_{S \in \br_{p}^{p_{k}}(\supp_{j-1}) } \nu_{p}(S) \lambda_{\tilde{S}}.\]
Another important point is the following:
\begin{assertion}
\label{assertion-r2}
$\rho_{p_{k}}  \neq 0$ for every $k=1,\ldots,\ell$.
\end{assertion}
\begin{proof}
Suppose, on the contrary, that $\rho_{p_{k}} =0$ for some $k$.
Consider the germ of logarithmic $1$-form at $p \in \supp_{j-1}$ defined as
\[ \zeta^{(j-1)}_{p} =
\sum_{S \in \br_{p}^{p_{k}}(\supp_{j-1}) }   \lambda_{\tilde{S}}  \frac{d f_{S}}{ f_{S}},\]
where  $f_{S} \in \cl{O}_{p}$ is the   equation of $S$ corresponding to the equation $f_{\tilde{S}} \in \cl{O}_{p_{k}}$ of $\tilde{S} = \sigma_{j}^{*}S$.
Since $\zeta^{(j-1)}_{p}$ has a main resonance at $p$ and all its poles share the same tangent cone, by \cite[Prop. 10]{canocorral2006},
 $\sigma_{j}$ is a non-dicritical blow-up for $\zeta^{(j-1)}_{p}$.
 Denote the germ of  $\sigma_{j}^{*}\zeta^{(j-1)}_{p}$ at $p_{k}$ by  $\zeta^{(j)}_{p_{k}}$.
Since $\rho_{p_{k}} =0$, the logarithmic $1$-forms $\zeta^{(j)}_{p_{k}}$ and  $\eta_{p_{k}}^{(j)}$ differ from each other by the germ of
 closed holomorphic $1$-form $\alpha_{p_{k}}$. Thereafter, $\zeta^{(j)}_{p_{k}}$ is also $\QQ_{p_{k}}^{(j)}$-logarithmic.
We thus have that $D_{p_{k}}$, the germ of $D$ at $p_{k}$, is not a pole for
 $\zeta_{p_{k}}^{(j)}$, although it   is  invariant.
For this to occur, $D_{p_{k}}$ should be a escape separatrix for the logarithmic $1$-form $\zeta^{(j)}_{p_{k}}$
and, as a consequence, $\varsigma_{j}^{*} D_{p_{k}}$ should touch $\varsigma_{j}^{-1}(p_{k}) \subset \D_{n}$ at a dicritical component.
However, this would mean that two  components  in $\dic_{n}$ meet each other, which is not allowed
by (D.3). This  contradiction  proves the assertion.
\end{proof}

In order to amalgamate  data at level $j-1$,
we chose $c_{1},\ldots,c_{\ell} \in \C^{*}$
such that
\begin{equation}
\label{eq-linear}
 c_{1}\rho_{p_{1}} + \cdots + c_{\ell}\rho_{p_{\ell}} = 0 ,
 \end{equation}
which is possible by the two assertions just proven.
Next, we modify the data of $\Lambda^{(j)}$ in the following way: at each $p_{k} \in D$,
for each branch $\tilde{S} \in \br_{p_{k}}(\supp_{j})$ we replace $\lambda_{\tilde{S}}$ by
$c_{k} \lambda_{\tilde{S}}$ and perform the same correction of residues at all points in the connected component of $\ndic_{j}$ containing $p_{k}$.  Once accomplished this adjustment at level $j$, we accordingly modify residues
at higher levels, obtaining new partial system of residues, that we denote by $\tilde{\Lambda}^{(j)}$.
Next, we define a partial system of residues  $\Lambda^{(j-1)}$ by
using information from $\tilde{\Lambda}^{(j)}$ at levels $j$ and higher and also descending
it to level $j-1$:
for a component
$S \subset \supp_{j-1}$ , we take $\lambda_{S}$ as the residue of $\tilde{S}= \sigma_{j}^{*}S$
in $\tilde{\Lambda}^{(j)}$. We clearly have a partial system of residues. In particular,
condition  (R.3) for
$\sigma_{j}$ is a consequence of \eqref{eq-linear}.

As for the $\QQ_{p}^{(j-1)}$-logarithmic $1$-form, where $\QQ_{p}^{(j-1)} = (\TT_{p}, \Lambda^{(j-1)}_{p})$,
in the same way as in the non-dicritical case,
 we define
\[ \eta^{(j-1)}_{p} =
\sum_{S \in \br_{p}(\supp_{j-1}) }   \lambda_{S}  \frac{d f_{S}}{ f_{S}}.\]
The germ of  $\sigma_{j}^{*}\eta^{(j-1)}_{p}$ at $p_{k}$ differs
from $c_{k}\eta^{(j)}_{p_{k}}$ by a germ of closed holomorphic $1$-form.
This implies that $\sigma_{j}^{*}\eta^{(j-1)}_{p}$ is faithful and, thus,
 $\QQ_{p_{k}}^{(j)}$-logarithmic.
On the other hand,
using equation \eqref{eq-linear}, we find that
 $\eta^{(j-1)}_{p}$ is $1$-faithful as a consequence of Proposition
\ref{prop-faithfull-amalgamation}, to be proven below
 in Section \ref{section-escape}.
In particular, it is dicritical for the blow-up $\sigma_{j}$.
We can thus conclude that $\eta^{(j-1)}_{p}$ is faithful, and, therefore, that it is
$\QQ_{p}^{(j-1)}$-logarithmic.
\end{proof}

\begin{example}
\label{ex-foliation-quadruplet}
{\rm In Examples \ref{ex-foliation-structure}, \ref{ex-foliation-separatrices} and \ref{ex-foliation-indices},
we built
$\TT(\F)  = (\Delta(\F),\Sigma(\F),\Upsilon(\F))$, a dicritical triplet   associated with a singular holomorphic foliation $\F$,
of generalized curve type, at $(\C^{2},0)$.
As a consequence of Theorem \ref{teo-quadruplet}, we can complete $\TT(\F)$ into a
dicritical quadruplet
$\QQ(\F) = (\TT(\F),\Lambda(\F)) = (\Delta(\F),\Sigma(\F),\Upsilon(\F),\Lambda(\F))$ and
find a $\QQ(\F)$-logarithmic $1$-form  $\eta_{\F}$. We say that  $\eta_{\F}$ is a logarithmic model for $\F$.
The existence of logarithmic models for germs of singular holomorphic foliations in the plane has already been established in
\cite{corral2003} and \cite{canocorral2006}.
}\end{example}


\section{Complex meromorphic functions}
\label{section-moromorphic-functions}
Let $h$ be a germ of meromorphic function at $(\C^{2},0)$.
Let  $\pi:(\tilde{M},\D) \to (\C^{2},0)$
be the reduction of singularities of the foliation $\HH$ defined by the level curves of $h$.
In particular, $\pi$ raises the
indeterminacy of $h$: there exists a germ of holomorphic map $\tilde{h}: \tilde{M} \to \Pe_{\C}^{1}$
in a neighborhood of $\D$ such that, outside $\D$, $\tilde{h} = h \circ \pi$.
We associate with $h$ the dicritical duplet of
the logarithmic $1$-form $\eta_{h} =  d h / h$: $\DD(h) = \DD(\eta_{h})$, where
$\DD(h) = (\Delta(h),\Sigma(h))$.
Actually, we have to assume that (S.3) is satisfied, as it will be the case below.
The typical fibers of $h$ have irreducible components whose equisingularity types
are determined by its \emph{indeterminacy structure}, i.e. the dicritical components of $\dic_{n}$.
On the other hand,  the isolated separatrices of $\sepndic_{0}$
are irreducible components of non-typical fibers.

\begin{example}
{\rm
Let $h= x^{p}/y^{q}$, with $p,q \in \Z_{+}$ relatively prime.
The combinatorics involved in
$\Delta(h)$ is the same of that of    Euclid's algorithm of the pair $(p,q)$. The sole
dicritical blow-up will be the last one, $\sigma_{n}$.
In other words, the dicritical structure depends on the residues of the
$1$-form $\eta_{h} = dh/h = p dx/x - q dy/y$.
If $p,q > 1$, then
$S_{1}: x=0$ and  $S_{2}: y=0$ are isolated separatrices.
On the other hand, if, for instance, $p> 1$ and $q=1$, then
$S_{1}$ is isolated and $S_{2}$ is dicritical.
}\end{example}

The central result  of this section concerns the construction of meromorphic functions
with prescribed indeterminacy structure, zeros and poles.
The indeterminacy structure of a meromorphic function $h$ can be specified by an
infinitesimal class  $\kappa$ of a union of components of some blow-up divisor, meant to identify the dicritical components of
the reduction of singularities of the foliation $\HH$ given by the levels of $h$.
We can then state:

\begin{maintheorem}
\label{teo-structure-meromporphic}
Let $\kappa$ be an infinitesimal class and $\sep$ be a finite set of branches of analytic curves at $(\C^{2},0)$.
Then there   exits a germ of meromorphic function $h$ at $(\C^{2},0)$ whose indeterminacy structure is given by $\kappa$ and such that
the set of all branches of    $h=0$ and $h=\infty$   contains  $\sep$.
\end{maintheorem}
\begin{proof} Let $\pi:(\tilde{M},\D) \to (\C^{2},0)$ be a sequence of blow-ups that is a minimal realization for the infinitesimal class   $\kappa$ and for the
equisingularity class $\varepsilon(\sep)$.
Build a   dicritical structure $\Delta$ over  $\pi$
with dicritical components determined by $\kappa$.
Next, complete $\sep$ into a finite subset $\sep_{0} \subset \br_{0}$ so that conditions (S.2) and (S.3) are  respected,
taking care   that also (S.4) is satisfied. Then $\sep_{0}$, with the decomposition given by $\Delta$, is the level 0 of a
configuration of separatrices $\Sigma$, and thus we have a dicritical duplet $\DD$ at $(\C^{2},0)$.
We obtain the theorem's statement by proving the following: there
exists a germ of meromorphic function $h$ such that
$\DD(h) \geq \DD$.
This result will follow from Theorem
\ref{teo-quadruplet} once we prove that we can associate  a system of
indices $\Upsilon$ having only negative rational indices at the singular points of the support of $\DD$  at the final level. We prove this next,
in Proposition \ref{prop-rational-index}, and assume it for now. Hence,  applying Theorem \ref{teo-quadruplet}, we can find
a system of residues $\Lambda$ that completes a dicritical quadruplet $\QQ= (\Delta,\Sigma,\Upsilon,\Lambda)$ as well as
a $\QQ$-logarithmic $1$-form $\eta$. An attentive look at the  induction in its proof shows that $\Lambda$ can be obtained
having only rational residues.
Indeed, its initialization  works by the application of
Proposition \ref{prop-residues-leveln}, where residues are obtained in $\Q^{*}$ once we depart from $\lambda_{D} \in \Q^{*}$
for the maximal element $D$ of each connected component $\A = \A_{n,k}^{\iota}$.
It is also easy to see that the amalgamation of data belonging to the general step $j$ can be done in such a
way that all residues are generated  in $\Q^{*}$.
Notably, in the dicritical case, in equation \eqref{eq-linear},
if all coefficients $\rho_{p_{k}}$ are in $\Q^{*}$, then we can get
a solution vector with coordinates in $\Q^{*}$.

With the writing of Definition \ref{def-Q-logarithmic}, after
 replacing some $f_{S}$ by
 $u f_{S}$, where $u \in \OO_{0}$ is a unit such that $ \lambda_{S} du / u =   \alpha$, and after
  multiplying by the
least common multiple of the   denominators  of the residues,
 we can write
\[ \eta = \sum_{S \in \sep_{0}} n_{S} \frac{d f_{S}}{f_{S}},\]
where $ n_{S} \in \Z^{*}$.
By setting
\[ h = \prod_{S \in \sep_{0}} f_{S}^{n_{S}} ,\]
we have that $\eta = \eta_{h} = d h /h$, proving the theorem.
\end{proof}

In order to complete the proof of Theorem \ref{teo-structure-meromporphic}, we  still have
to show the possibility of
assigning   negative rational indices at the final level. For this purpose,
we shall explore the combinatorics dictated by  rules (I.2), (I.3) and (I.4).
Let us then
associate with $\D = \pi^{-1}(0)$  a weighted graph $\gr(\D)$, in which
each irreducible component $D$ corresponds to a vertex, still denoted by $D$, whose weight is
$c(D) = D \cdot D$,   with edges marking intersections
of components. The graph $\gr(\D)$ is a tree whose intersection matrix is negative definite by \cite{laufer1971}.
More generally, if $\A \subset \D$ is a connected union of irreducible components of $\D$,
then $\gr(\A)$ is a subtree with negative definite intersection matrix.
Indeed, up to renumbering the vertices of $\gr(\D)$, the intersection matrix  of $\gr(\A)$ can be seen as a diagonal
block of that of $\gr(\D)$.

We complete the proof of  Theorem \ref{teo-structure-meromporphic} with the following:

\begin{prop}
\label{prop-rational-index}
Let $\DD = (\Delta,\Sigma)$ be a dicritical duplet at $(\C^{2},0)$ of height $n = h(\Delta)$. Then there exists a
system of indices $\Upsilon$ associated with $\DD$ such that $I_{p}(S) \in \Q_{-}$
for every singular point $p \in \supp_{n}$ and  every $S \in \br_{p}(\supp_{n})$.
\end{prop}
\begin{proof}
Since we can work separately with each connected component of $\ndic$,
 we fix some $\A = \A_{n,k}^{\iota} \subset \ndic$ and associate its weighted graph $\gr(\A)$.
Observe that a small perturbation in the weights of $\gr(\A)$  does not change the negative definiteness of its
intersection matrix. More precisely, there is $\epsilon >0$ such that if $\widetilde{\gr}(\A)$ is the weighted graph  obtained from $\gr(\A)$
by replacing each weight  $c(D)$ by some $\tilde{c}(D)$ satisfying
 $|\tilde{c}(D) - c(D)| < \epsilon$, then the intersection matrix of $\widetilde{\gr}(\A)$ is negative definite.

Choose
$S \in \sepndic_{0}$ such
that $\pi^{*}S$ intersects $\A$ at a point   $p \in D_{0} = D^{\pi}(S)$.
This $S$ exists by (S.2).
Let us order the vertices of $\gr(\A)$ in such a way that $D_{0}$ is a maximal element. Following this order,  stratify the vertices of
$\gr(\A)$ in levels $j=1,\ldots,s$, where level  $s$
corresponds to the maximal vertex $D_{0}$, level $s-1$ to its immediate predecessors and so on.
Next,  for each component $D \subset \A$ and for each trace singularity $q \in D$, with $q \neq p$, choose for $I_{q}(D)$ a sufficiently small value in $\Q_{-}$
in such a way that their sum over $D$ satisfies $\left|\sum  I_{q}(D) \right| < \epsilon$.
Define $\tilde{c}(D)= c(D) - \sum  I_{q}(D)$. In this way,
$\widetilde{\gr}(\A)$ has a   negative definite intersection matrix.
The idea now is to  use  \cite[Prop. 3.3]{camacho1988} in a context where the sum of indices over a component $D$ is $\tilde{c}(D)$,
in order to show that all indices, for all corners and for the point $p$, are in $\Q_{-}$.
The argument goes by induction, with the initialization and the general step treated with the same argument.
Let $1 \leq j \leq s$ and suppose that all corners in components at levels lower than $j$ have indices in $\Q_{-}$ (an empty condition,
if $j=1$).
If $D$ is a component at level $j$, denote by $\widetilde{\gr}(\A_{D})$ the subtree of $\widetilde{\gr}(\A)$ formed by $D$ and all its predecessors.
 $\widetilde{\gr}(\A_{D})$ also has a negative definite intersection matrix. If $j<s$,
let $D'$ be the successor of $D$ and denote $p' = D \cap D'$. If $j=s$, set $D'= S$ and $p'=p$.
Denote by $q_{1},\ldots,q_{r}$  the corners of $D$ with its predecessors, if there are any. Then the induction hypothesis  and (I.4) gives that $I_{q_{i}}(D) \in \Q_{-}$
for   $i=1,\ldots,r$.
The sum of indices over $D$ equaling $\tilde{c}(D)$  gives  that $I_{p'}(D) \in \Q$.
Finally, regarding
$\widetilde{\gr}(\A_{D})$ as a combinatorial portrait of an invariant divisor with only one trace singularity at $p'$,
with  sums of indices given by $\tilde{c}$ in place of $c$, an application
of \cite[Prop. 3.3]{camacho1988} gives that  $I_{p'}(D) \not\in  \R_{\geq 0}$. Thus, $I_{p'}(D) \in \Q_{-}$ and also
$I_{p'}(D') \in \Q_{-}$.
The proposition is  proved once
induction reaches level $j=s$.
\end{proof}

\begin{remark}
{\rm  In Theorem
\ref{teo-structure-meromporphic} the set of branches of zeroes and poles of $h$ contains    $\sep$ as a proper subset, in general. As it is clear in the beginning
of its proof, we have  to include in $\sep$ some branches in order to obtain a set  where
conditions (S.2) and (S.3) are valid. The necessity of  condition (S.3) is apparent: if $D \subset \D$ is a dicritical
component of valence one, then the lift $\tilde{h}$ of $h$ to $\tilde{M}$, restricted  $D \simeq \Pe^{1}_{\C}$, must have a zero and a pole.
One of them being given by the component of $\overline{\D \setminus D}$ intersecting $D$, the other must be given by a
separatrix touching $D$.
}\end{remark}


\section{Real logarithmic models}
\label{section-real-models}

\subsection{Symmetric dicritical quadruplets}

Let $J:(x,y) \in \C^{2} \mapsto (\bar{x},\bar{y}) \in \C^{2}$ be the canonical
  anti-holomorphic involution defined by the complex conjugation. The \emph{real trace} of $\C^{2}$, the fixed set
 of $J$, is identified with $\R^{2}$.
Points are  said to be \emph{real} when they belong to the real trace, and \emph{non-real} when they do not.
For $f \in \OO_{0}$, let $f^{\vee} \in \OO_{0}$
be such   that
$f^\vee(x,y) = \xbar{f(J(x,y))}$
for every $(x,y)$ near $0 \in \C^{2}$.
We say that $f \in \OO_{0}$ is \emph{real} if  $f = f^{\vee}$. This is equivalent to $f$ having
a Taylor series expansion at $0 \in \C^{2}$  with only real coefficients or to   $f$ assuming only real values over the real trace.
For a branch $S \in \br_{0}$ having $f \in \OO_{0}$ as equation, we denote by $S^{\vee} \in \br_{0}$
the branch with equation $f^{\vee}$, whose trace is the $J$-image of the trace of $S$.
We say that a branch $S \in \br_{0}$ is \emph{real} if  $S=S^{\vee}$.  This  happens if and only if
 $S$ has a real $f \in \OO_{0}$ as equation.

A germ of holomorphic  $1$-form $\omega = P dx + Q dy$ at $(\C^{2},0)$ is \emph{symmetric} if $P,Q \in \OO_{0}$ are real.
In this case, the germ of holomorphic foliation $\F$ defined by $\omega$, also said to be  symmetric, is invariant by $J$:
if $U$ is a $J$-invariant neighborhood  of $0 \in \C^{2}$ where $\omega$ is realized, and $p \in U$ is a regular point,
denoting by $L_{p}$  the leaf through  $p$, then $J(L_{p})$ is the leaf through $p^{\vee} = J(p)$.
In particular, if $p \in \R^{2}$, then
$L_{p} = J(L_{p})$. Thus, $L_{p}|_{\R^{2}}$ is a local curve in $\R^{2}$ (of topological dimension one), which is
the local leaf of the foliation  $\F_{\R}$ defined by $\omega_{\R}$, the restriction of $\omega$ to the real trace.
Note also that, if $\F$ is symmetric, $S$ is a separatrix at a real point $p$ if and only if $S^{\vee}$ is.

If $\sigma: (\tilde{\C^{2}},D) \to (\C^{2},0)$ is a punctual blow-up, then there is a unique
continuous involution $J':(\tilde{\C^{2}},D) \to (\tilde{\C^{2}},D)$
such that $\sigma(J'(p)) = J(\sigma(p))$ for every $p \in \tilde{\C}^{2}$. 
Its fixed set,  $\tilde{\R}^{2}$, is the \emph{real trace} of $\tilde{\C^{2}}$.
Evidently, the involution $J$ will be lifted by successive blow-ups provided their centers  lie in the real trace.
More generally, we say that a
 sequence $\pi:(\tilde{M},\D) \to (\C^{2},0)$ of quadratic blow-ups is \emph{symmetric}
if the anti-holomorphic involution $J$
has a continuous lift  $\tilde{J}$ to $\tilde{M}$.
Loosely speaking, this happens if, among the individual blow-ups factoring $\pi$,
those with non-real centers are even in number and  paired by conjugation.

A dicritical structure $\Delta$ of height $n = h(\Delta)$ at $(\C^{2},0)$ is \emph{symmetric} if its underlying sequence of
blow-ups is symmetric and, for every component $D \subset \D_{n}$, we have that $D \subset \dic_{n}$ if and only if
$D^{\vee} = \tilde{J}(D)  \subset \dic_{n}$.
A configuration of separatrices $\Sigma$ at $(\C^{2},0)$ is \emph{symmetric} if it is framed on a symmetric dicritical
structure $\Delta$ and, also, $\sep_{0}$ is invariant by $J$. Clearly, this also gives  that
$\sep_{n}$ is invariant by $\tilde{J}$ and, as a consequence,
it holds that $S \in \sepdic_{0}$ if and only if $S^{\vee} \in \sepdic_{0}$.
By (S.4), if $p \in \D_{n}$ is a real point and $S \in \sep_{n}$ is a separatrix through $p$, then $S$ is real.
In particular, if $p \in \supp_{n}$ is a real singular point, then both branches of $\supp_{n}$ at $p$ are real.
We  say that $\DD = (\Delta,\Sigma)$ is a \emph{symmetric dicritical duplet}.

A system of indices $\Upsilon$ associated with a symmetric dicritical duplet $\DD$ is \emph{symmetric} if
 whenever $p \in \supp_{n}$   and   $S \in \br_{p}(\supp_{n})$, then $I_{p^{\vee}}(S^{\vee})=\xbar{I_{p}(S)}$.
In particular, if $p \in \supp_{n}$  is real singular point, then $I_{p}(S) \in \R^{*} \setminus \Q_{+}$ for each branch $S \in \br_{p}(\supp_{n})$.
We say that $\TT =(\Delta,\Sigma,\Upsilon)$ is a \emph{symmetric dicritical triplet}.
Finally, a $\emph{symmetric}$   system of residues $\Lambda$ for the symmetric dicritical duplet $\DD =(\Delta,\Sigma)$ is one for which
 $\lambda_{S^{\vee}} = \xbar{\lambda_{S}}$ whenever $S \in \sep_{0}$.
Thus,  $\lambda_{S} \in \R^{*}$ whenever $S$ is a real separatrix.
The application of (R.1), (R.2) and (R.3) gives that $\lambda_{S^{\vee}} = \xbar{\lambda_{S}}$
for every component $S \subset \supp_{n}$.
If $\Upsilon$ and $\Lambda$ are consistent  symmetric $\DD$-systems of indices and residues, then we say
that  $\QQ = (\Delta,\Sigma,\Upsilon,\Lambda)$ is a \emph{symmetric dicritical quadruplet}.

We say that a logarithmic $1$-form as in  \eqref{eq-log-form} is \emph{symmetric} if its set of poles $\sep \subset \br_{0}$ is symmetric, as well as the corresponding residues and the closed holomorphic $1$-form
$\alpha$.
We can restate Theorem \ref{teo-quadruplet} in  the context of symmetric multiplets.  The very same proof presented in Section \ref{section-existence-quadruplet} works here
and we leave  to the reader
its step-by-step verification.
\begin{teo}
\label{teo-quadruplet-real}
Let  $\TT =(\Delta,\Sigma,\Upsilon)$ be a symmetric dicritical triplet. Then there exists:
\begin{enumerate}
  \item a symmetric system of
residues $\Lambda$ such that $\QQ = (\TT,\Lambda) = (\Delta,\Sigma,\Upsilon,\Lambda)$ is a symmetric dicritical quadruplet;
  \item  a  $\QQ$-logarithmic symmetric $1$-form $\eta$.
\end{enumerate}
\end{teo}

In the next subsection, we will handle   abstract multiplets formed by real objects that
will model the desingularization, separatrices and indices of germs of real analytic vector fields
at $(\R^{2},0)$.
The complexification of these real objects will eventually give rise to symmetric multiplets.
In order to achieve this, we  first  have to consider the following:

\begin{ddef}
\label{def-quasi-dicritical-triplet}
{\rm
 A \emph{quasi  dicritical triplet}  at $(\C^{2},0)$, or \emph{q-dicritical triplet} for short,  is the object
$\TT^{\aleph}    = (\Delta^{\aleph} ,\Sigma^{\aleph} ,\Upsilon^{\aleph} )$ composed by:
\begin{itemize}
\item a dicritical structure $\Delta^{\aleph}$;
\item a configuration of separatrices $\Sigma^{\aleph}$   with conditions (S.2) and (S.3) removed;
\item a system of indices  $\Upsilon^{\aleph}$   with condition (I.3) deleted.
\end{itemize}
}\end{ddef}
The   q-dicritical triplet $\TT^{\aleph}$ is   \emph{symmetric} if all its elements are   invariant by conjugation.
In this case, we say that a symmetric dicritical triplet  $\TT     = (\Delta,\Sigma,\Upsilon )$ \emph{extends}  $\TT^{\aleph}$ if:
\begin{itemize}
\item The  sequence of blow-ups subjacent   to $\TT$  is obtained from that of  $\TT^{\aleph}$ by additional blow-ups at
non-real points; $\Delta$ and $\Delta^{\aleph}$ provide the same classification as dicritical or non-dicritical for the common
  components.
\item If $\sep_{0}^{\aleph}$ and $\sep_{0}$ denote the set of separatrices at level $0$ of $\Sigma^{\aleph}$ and $\Sigma$,  respectively, then
$\sep_{0}^{\aleph} \subset \sep_{0}$ and $\sep_{0} \setminus \sep_{0}^{\aleph}$ contains only non-real separatrices.
\item For separatrices in $\sep_{0}^{\aleph}$, indices assigned by $\Upsilon$ and $\Upsilon^{\aleph}$ are the same.
\end{itemize}
We can always extend a symmetric q-dicritical triplet, as shown in the following result.

\begin{prop}
\label{prop-quasi-triplet}
Let $\TT^{\aleph}    = (\Delta^{\aleph} ,\Sigma^{\aleph} ,\Upsilon^{\aleph} )$ be a   symmetric q-dicritical triplet at $(\C^{2},0)$.
Then there exists a symmetric dicritical triplet  $\TT     = (\Delta,\Sigma,\Upsilon)$ that extends  $\TT^{\aleph}$.
All indices in $\Upsilon$ can be obtained in $\R$.
\end{prop}
\begin{proof} We take $\Delta = \Delta^{\aleph}$. Our task is to add non-real separatrices to $\Sigma^{\aleph}$ and assign
them indices
in such a way that rules (S.2), (S.3) and (I.3) are observed.
Let $\A = \A_{n,k}^{\iota}$ be a connected component of $\ndic_{n}$, the invariant part of the final level of $\Delta$. Let $D \subset \A$ be a  component.
Define $c^{\aleph}(D) = \sum I_{p}(D)$, where $p$ runs over the singular points of $D$ with respect to $\TT^{\aleph}$, i.e. its corners  and
points where separatrices in $\sep_{n}^{\aleph}$ meet $D$. If $c^{\aleph}(D) = c(D)$, we do nothing.
If
$c(D) - c^{\aleph}(D) \in \R \setminus \Q_{\geq 0}$,
we choose a non-real point $p \in D$, take $S$ any smooth branch at $p$ transverse to $D$, add it to $\sepndic_{n}$,
with index  $I_{p}(S) = 2/(c(D) - c^{\aleph}(D))$ getting, as a consequence,
$I_{p}(D) = (c(D) - c^{\aleph}(D))/2$.  We also include $S^{\vee}$  in $\sepndic_{n}$, with index $I_{p^{\vee}}(S^{\vee}) = I_{p}(S)$.
Now, if $c(D) - c^{\aleph}(D) \in  \Q_{+}$, then we choose
$b_{1},b_{2} \in \R \setminus \Q$ such that $b_{1} + b_{2} = (c(D) - c^{\aleph}(D))/2$
and, also, two distinct, non-conjugate,   non-real points $p_{1}, p_{2} \in D$.
Take $S_{1}$ and $S_{2}$ two smooth branches at, respectively,  $p_{1}$ and $p_{2}$, transverse to $D$,
and add them  to $\sepndic_{n}$,
with indices  $I_{p_{1}}(S_{1}) = 1/b_{1}$ and  $I_{p_{2}}(S_{2}) = 1/b_{2}$.
Then, $I_{p_{1}}(D) + I_{p_{2}}(D) = b_{1} + b_{2}  = (c(D) - c^{\aleph}(D))/2$.
Next, also include
$S_{1}^{\vee}$ and $S_{2}^{\vee}$ to $\sepndic_{n}$
with indices  $I_{p_{1}^{\vee}}(S_{1}^{\vee}) = 1/b_{1}$ and  $I_{p_{2}^{\vee}}(S_{2}^{\vee}) = 1/b_{2}$.

The object
$\TT     = (\Delta,\Sigma,\Upsilon)$ constructed so far satisfies (I.3) for all components $D \in \A$.
It does not comply with (S.2) if $\sep_{n}^{\aleph}$  has no separatrices touching $\A$ and if
$c^{\aleph}(D) = c(D)$ for every $D \subset \A$. If this is so, fix $D \subset \A$, pick two distinct, non-conjugate,  non-real points $p_{1}, p_{2} \in D$, then   choose
$b_{1},b_{2} \in \R \setminus \Q$ such that $b_{1} + b_{2} = 0$
 and repeat the construction that closes the previous paragraph.
The above procedure,    applied to all connected components $\A_{n,k}^{\iota} \subset \ndic_{n}$, results in the validity of (I.3) and (S.2).
To finish the proof, if  for some $D \subset \dic_{n}$ with $\val(D) = 1$ condition (S.3) is not verified by a separatrix of $\sep_{n}^{\aleph}$, we add
to $\sepdic_{n}$ a pair of non-real symmetric separatrices, touching $D$ transversally at pair of different trace points, and give them both  zero as indices, doing the corresponding symmetric intervention in $D^{\vee}$ if $D$ is non-real.
\end{proof}

\subsection{Real analytic 1-forms}
Suppose   that $(x,y)$ are now
coordinates in $\R^{2}$. Let $C_{0}^{\omega}$ denote the ring of germs of
real analytic functions with real values at $(\R^{2},0)$.
Let $\omega_{\R} = P(x,y) dx + Q(x,y)dy$ be a germ of real analytic $1$-form at $(\R^{2},0)$, where
$P,Q \in C_{0}^{\omega}$
are relatively prime  non-units.
Then, $\omega_{\R}$ defines, near $0 \in \R^{2}$, a real analytic foliation $\F_{\R}$ with isolated singularity at the origin.
We denote by $\omega$ the complexification of $\omega_{\R}$, which amounts to regarding $(x,y)$ as coordinates of $\C^{2}$ and
$P,Q$ as elements of $\OO_{0}$. Then, according to the definition in the previous section, $\omega$ is a symmetric $1$-form,
 with an isolated singularity at $0 \in \C^{2}$, defining a singular holomorphic foliation
$\F$   invariant by the involution $J$. Clearly, viewing $\R^{2}$ as the real trace of $\C^{2}$, we have that
 $\F_{\R} = \F|_{\R^{2}}$.

The foliation $\F_{\R}$ has a \emph{reduction of singularities} by a sequence of real quadratic blow-ups, say $\pi_{\R}:(\tilde{M}_{\R},\D_{\R}) \to (\R^{2},0)$.
This means that, for the strict transform  foliation $\tilde{\F}_{\R} = \pi_{\R}^{*}\F_{\R}$, we can obtain the real equivalent of the four items listed in Example \ref{ex-foliation-structure},
with the additional condition that real eigenvalues are associated  with simple singularities.
In order to obtain $\pi_{\R}$, we can consider, for instance,
the minimal reduction of singularities for $\F$, say $\pi:(\tilde{M},\D) \to (\C^{2},0)$,  and
take $\tilde{M}_{\R}$ as the real trace of
$\tilde{M}$, determined by the lift   of $J$, getting also $\D_{\R} = \D \cap \tilde{M}_{\R}$
and   $\pi_{\R} = \pi|_{\tilde{M}_{\R}}$.

The following definition
associates with a  sequence of real quadratic blow-ups  an abstract triplet composed
only by real objects.
\begin{ddef}
\label{def-real-quasi-dicritical-triplet}
{\rm
 A \emph{real quasi  dicritical triplet}, or \emph{real q-dicritical triplet},  framed on a sequence of real quadratic blow-ups $\pi_{\R}$  is the object
$\TT^{\aleph}_{\R}    = (\Delta^{\aleph}_{\R} ,\Sigma^{\aleph}_{\R} ,\Upsilon^{\aleph}_{\R})$ whose components are:
\begin{itemize}
\item $\Delta^{\aleph}_{\R}$ is the structure of real divisors produced by the sequence of blow-ups $\pi_{\R}$;
\item $\Sigma^{\aleph}_{\R}$  is formed by real analytic branches;
\item $\Upsilon^{\aleph}_{\R}$ includes only real
indices.
\end{itemize}
They satisfy the set of axioms that defines  a dicritical triplet, except for  (S.2), (S.3) and (I.3).
}\end{ddef}
Evidently,
the complexification of the constituents of $\TT^{\aleph}_{\R}$ gives readily
a symmetric q-dicritical triplet $\TT^{\aleph}$,   framed on the sequence of
complex quadratic blow-ups  $\pi$, the complexification of $\pi_{\R}$.

Following the definition in \cite{risler2001},  $\F_{\R}$ is of \emph{real generalized curve} type if $\tilde{\F}_{\R}$
has no real algebraic saddle-node singularities, i.e. simple singularities with one zero eigenvalue. This is equivalent to asking that $\tilde{\F}$ has no  saddle-node singularities over the real trace of $\tilde{M}$.
With a foliation of this kind, we can associate  a
 real q-dicritical triplet
 $\TT^{\aleph}_{\R}(\F_{\R})= (\Delta^{\aleph}_{\R}(\F_{\R}),\Sigma^{\aleph}_{\R}(\F_{\R}),\Upsilon^{\aleph}_{\R}(\F_{\R}))$,
including in $\Sigma^{\aleph}_{\R}(\F_{\R})$ all isolated separatrices and a finite number of dicritical separatrices.
It gives rise, by complexification, to a symmetric q-dicritical triplet
 $\TT^{\aleph}(\F_{\R})= (\Delta^{\aleph}(\F_{\R}),\Sigma^{\aleph}(\F_{\R}),\Upsilon^{\aleph}(\F_{\R}))$.

We say that a germ of $1$-form $\eta_{\R}$ at $(\R^{2},0)$ is \emph{real logarithmic} if it has a writing as in \eqref{eq-log-form1}, where each $f_{i} \in  C_{0}^{\omega}$ is irreducible, the residues
$\lambda_{i}$ are real and $\alpha$ is a germ of real analytic $1$-form.
Note that the zero set of some of the $f_{i}$ may degenerate to the origin. When $f_{i} = 0$ is one dimensional, we say that it defines  a \emph{real pole}
 $S_{i}$ of $\eta_{\R}$.
The notion of $1$-faithful logarithmic $1$-form, set in Definition \ref{def-faithful}, extends unequivocally to  the real context.
We then say that a real logarithmic $1$-form is real faithful if it is
$1$-faithful along the reduction of singularities of $\LL_{\R}$,
  the germ of real analytic foliation at $(\R^{2},0)$ defined by $\eta_{\R}$.

Following the same steps of the construction of a dicritical triplet for a complex logarithmic $1$-form, preceding Definition \ref{def-system-residues},
we can build a real q-dicritical triplet
$\TT^{\aleph}_{\R}(\eta_{\R})    = (\Delta^{\aleph}_{\R}(\eta_{\R}) ,\Sigma^{\aleph}_{\R}(\eta_{\R}),\Upsilon^{\aleph}_{\R}(\eta_{\R}))$
for a real logarithmic $\eta_{\R}$.
Briefly, $\Delta^{\aleph}_{\R}(\eta_{\R})$ is based on the real reduction of singularities of $\LL_{\R}$, $\Sigma^{\aleph}_{\R}(\eta_{\R})$ includes
all real  poles of $\eta_{\R}$ along with all isolated separatrices of $\LL_{\R}$, and $\Upsilon^{\aleph}_{\R}(\eta_{\R})$ collects all Camacho-Sad indices
of $\LL_{\R}$.
Next, the notion of dominance for dicritical duplets can be transposed in a straightforward way to real q-dicritical duplets, and, with it,
the concepts of escape point and escape separatrix. Keeping our notation, the writing
$\DD^{\aleph \, '}_{\R}   \geq \DD^{\aleph}_{\R}$
means that $\DD^{\aleph \, '}_{\R} = (\Delta^{\aleph \, '}_{\R} ,\Sigma^{\aleph \, '}_{\R})  $ dominates $\DD^{\aleph}_{\R}  = (\Delta^{\aleph}_{\R} ,\Sigma^{\aleph}_{\R})$.
\begin{ddef}
\label{def-real-log-model}
 {\rm
Let $\TT^{\aleph}_{\R}    = (\Delta^{\aleph}_{\R} ,\Sigma^{\aleph}_{\R} ,\Upsilon^{\aleph}_{\R})$ be a real q-dicritical
triplet at $(\R^{2},0)$.
 A germ of real logarithmic $1$-form
$\eta_{\R}$ is a \emph{real logarithmic model} for  $\TT^{\aleph}_{\R}$ if:
\begin{enumerate}[label=(\arabic*)]
\item the set of real poles of $\eta_{\R}$ is $\sep^{\aleph}_{0}$,  the set of  separatrices of
 $\Sigma^{\aleph}_{\R}$ at level $0$;
\item $\DD^{\aleph}_{\R}(\eta_{\R}) \geq \DD^{\aleph}_{\R}$;
\item $\eta_{\R}$ is real faithful;
\item indices for separatrices in $\sep^{\aleph}_{0}$ are the same, for $\Upsilon^{\aleph}_{\R}$ and for $\Upsilon^{\aleph}_{\R}(\eta_{\R})$.
\end{enumerate}
}\end{ddef}
The real logarithmic model is \emph{strict} if  $\DD^{\aleph}_{\R}(\eta_{\R}) = \DD^{\aleph}_{\R}$.
The following result generalizes, to the dicritical case, the main theorem in \cite{corral2012}:
\begin{maintheorem}
\label{teo-real-logmodel}
Let $\TT^{\aleph}_{\R}    = (\Delta^{\aleph}_{\R} ,\Sigma^{\aleph}_{\R} ,\Upsilon^{\aleph}_{\R})$ be a real q-dicritical
triplet at $(\R^{2},0)$. Then there exists a germ of real logarithmic $1$-form $\eta_{\R}$   that is a
real logarithmic model for $\TT^{\aleph}_{\R}$.
\end{maintheorem}
\begin{proof}
We complexify $\TT^{\aleph}_{\R}$, obtaining
 the symmetric q-dicritical triplet $\TT^{\aleph}$.
 By Proposition \ref{prop-quasi-triplet}, $\TT^{\aleph}$ can be completed into a
symmetric dicritical triplet $\TT$ with real indices.
As a consequence of Theorem \ref{teo-quadruplet-real},
there are a symmetric system of
residues $\Lambda$, with real residues,  such that $\QQ = (\TT,\Lambda)$ is a symmetric dicritical quadruplet,
and  a symmetric $\QQ$-logarithmic   $1$-form $\eta$, whose residues are real. If $\eta_{\R}$ is the restriction of $\eta$ to the real trace of $\C^{2}$,
then $\eta_{\R}$ is a real logarithmic model for $\TT^{\aleph}_{\R}$.
\end{proof}
In Definition \ref{def-real-log-model}, when  $\TT^{\aleph}_{\R}  = \TT^{\aleph}_{\R}(\F_{\R})$ for some germ of singular real analytic  foliation $\F_{\R}$ of real generalized curve type,
we say that  $\eta_{\R}$ is a \emph{real logarithmic model} for $\F_{\R}$. The logarithmic model is \emph{strict} if
$\eta_{\R}$   is a strict
real logarithmic model for some real q-dicritical triplet  $\TT^{\aleph}_{\R}(\F_{\R})$ associated with $\F_{\R}$.
For germs of real analytic foliations, we can assure the existence of strict real logarithmic models:
\begin{maintheorem}
\label{teo-strict-logmodel-real}
Let $\F_{\R}$ be a
germ of singular real analytic  foliation at $(\R^{2},0)$   of real generalized curve type.
Then there exists a strict
real logarithmic model for $\F_{\R}$.
\end{maintheorem}
\begin{proof} The existence of a logarithmic model is a consequence of Theorem \ref{teo-real-logmodel}.
The possibility of obtaining a
strict real logarithmic model is justified in the next section, by Proposition \ref{prop-withour-real-escape}
and    the comments following it.
\end{proof}

As we did for complex sequences of blow-ups in Section \ref{section-dicritical-structure}, we can
define \emph{real infinitesimal classes} for sequences of real quadratic blow-ups. In this case, they are denoted by
$\kappa_{\R}$.
We can then apply the machinery of this section   to the construction of real meromorphic functions, yielding the following real version of
Theorem \ref{teo-structure-meromporphic}:

\begin{maintheorem}
\label{teo-structure-meromporphic-real}
Let $\kappa_{\R}$ be a real infinitesimal class and $\sep$ be a finite set of branches of real analytic curves at $(\R^{2},0)$.
Then there   exits a germ of real meromorphic function $h$ at $(\R^{2},0)$ whose indeterminacy structure is given by $\kappa_{\R}$ and such that
the set of all branches of    $h=0$ and $h=\infty$   equals  $\sep$.
\end{maintheorem}
\begin{proof} Information from $\kappa_{\R}$ and $\sep$ define a real q-dicritical duplet
$\DD^{\aleph}_{\R}    = (\Delta^{\aleph}_{\R} ,\Sigma^{\aleph}_{\R})$. We complexify  it
and add a symmetric set of complex separatrices in order that conditions (S.2) and (S.3) are accomplished,   obtaining a dicritical duplet
$\DD    = (\Delta,\Sigma)$.
Now we proceed as in Section \ref{section-moromorphic-functions}.
By Proposition \ref{prop-rational-index}, we can produce a system of indices $\Upsilon$ associated with $\DD$ having indices
in $\Q_{-}$, and we do it in a symmetric way with respect to the involution $J$.
As in Theorem \ref{teo-structure-meromporphic}, having as reference the symmetric dicritical triplet
$\TT    = (\Delta,\Sigma,\Upsilon)$,  we build a germ of  meromorphic function $h$ at $(\C^{2},0)$, which is also symmetric with respect to $J$.
The restriction of $h$ to the real trace $\R^{2}$ is the desired germ of real meromorphic function.
\end{proof}


\section{Escape set and escape separatrices}
\label{section-escape}

In the construction of logarithmic models, as   byproducts,  some separatrices
not originally modelled  by the dicritical quadruplet appear.
If $\QQ$ is a dicritical quadruplet and $\eta$ is a $\QQ$-logarithmic $1$-form, it turns out that the reduction of
singularities of the singular foliation $\LL$ defined by $\eta$ may be longer than the underlying sequence of blow-ups of $\QQ$. The additional blow-ups
start at centers called escape points, which are  dicritical  points for $\QQ$, outside   its support, where
the strict transform of $\eta$ is closed holomorphic, having thus a holomorphic first integral.
Escape points   give rise to some isolated separatrices for $\LL$, accordingly called
escape separatrices.
In this section we shall present a description of  these objects.
In the real case, escape points   can be eliminated  in
a process   called logarithmic modification, by which  additional dicritical separatrices are included,
  without affecting the data conveyed by the original  dicritical quadruplet.

\subsection{Escape function and faithful $1$-forms}

 Let $\eta$ be a  logarithmic $1-$form at $(\C^{2},0)$, written as in \eqref{eq-log-form}. We can suppose that
 the holomorphic closed part $\alpha$ has been incorporated into one of the equations of the polar set, in such a way that
\[ \eta = \lambda_{1} \frac{d f_{1}}{f_{1}} + \cdots +  \lambda_{r} \frac{d f_{r}}{f_{r}},\]
where  $r \geq 2$ and, for $i=1,\ldots,r$, $f_{i} \in \cl{O}_{0}$ is    irreducible with order $ \nu_{i}  = \nu_{0}(f_{i})$
and $\lambda_{i}  \in \C^{*}$.
Set
$\lambda_{0} = \sum_{i=1}^{r} \nu_{i} \lambda_{i}$, the weighted balance of residues of $\eta$.
Denote by $\LL$
the germ of holomorphic foliation  at $(\C^{2},0)$ defined by $\eta$.
This foliation is also induced by the holomorphic $1$-form  $\omega = f \eta$, where $f  = f_{1} \cdots f_{r}$,
which may have a one dimensional singular set.
For the quadratic blow-up $\sigma: (\tilde{\C}^{2},D) \to (\C^{2},0)$, we consider
coordinates $(x,t) \in \C^{2}$ such that  $\sigma(x,t) = (x,xt)$, in which $D = \sigma^{-1}(0)$ has $x=0$ as an equation.
 Then
the \emph{divided blow-up} of $\omega$  (even in the non-isolated singularity case) is
\begin{equation}
\label{eq-divided-blowup}
 \text{(i) non-dicritical:} \ \ \tilde{\omega} = \sigma^{*}\omega / x^{m_{0}} \qquad \text{or} \qquad
 \text{(ii) dicritical:} \ \ \tilde{\omega} = \sigma^{*}\omega / x^{m_{0}+1},
\end{equation}
where $m_{0} = \nu_{0}(\omega)=  \nu_{0}(f \eta)$ (see, for instance, \cite{camacho1984}).
The $1$-form $\tilde{\omega}$, which induces $\tilde{\LL} = \sigma^{*} \LL$  in the coordinates $(x,t)$, is  holomorphic and  does not  contain $D$ in its singular set.

Recall, from   Definition \ref{def-faithful}, that
$\eta$ is $1$-faithful if and only if $m_{0} = \nu_{0}(\omega) = \nu_{0}(f \eta) = m-1$, where  $m = \nu_{0}(f) = \nu_{1} + \cdots + \nu_{r}$.
This is equivalent to the fact that, in
\[ \omega = f \eta =  \sum_{i=1}^{r} \lambda_{i} f_{1} \cdots \widehat{f_{i}} \cdots f_{r} d f_{i} , \]
the initial term is precisely that obtained by operating with the initial terms of the equations $f_{i}$.
In other words, it is expressed exclusively in terms of the residues $\lambda_{i}$ and of the equations of the tangent cones of the functions $f_{i}$.

By a linear change of coordinates,
we can assume that none of the functions $f_{i}$   contains  the $y$-axis in its tangent
cone. We  write, after possibly multiplying   by   non-zero constants,
\[ f_{i}(x,y) = (y -\alpha_{i}x)^{\nu_{i}} +  g_{i}(x,y), \]
where
$\alpha_{i} \in \C$ and  $g_{i}(x,y) \in \OO_{0}$ is such that $\nu_{0}(g_{i}) \geq \nu_{i}+1$.

We have
\[f_{i}(x,xt) = x^{\nu_{i}}\tilde{f}_{i}(x,t)= x^{\nu_{i}} \left((t- \alpha_{i})^{\nu_{i}} + x \tilde{g}_{i}(x,t) \right),\]
where $\tilde{g}_{i}(x,t) = g_{i}(x,xt)/x^{\nu_{i}+1}$.
It follows readily that
\begin{equation}
\label{eq-sigma*eta}
\tilde{\eta} = \sigma^{*} \eta = \lambda_{0} \frac{dx}{x} +
\lambda_{1} \frac{d \tilde{f}_{1}}{\tilde{f}_{1}} + \cdots +  \lambda_{r} \frac{d \tilde{f}_{r}}{\tilde{f}_{r}}.
\end{equation}

If we set    $\tilde{f} = \tilde{f}_{1} \cdots \tilde{f}_{r}$, we have
\begin{equation}
\label{eq-sigma*omega}
 \sigma^{*} \omega =  \sigma^{*}  ( f  \eta )= x^{m} \tilde{f}  \tilde{\eta}.
 \end{equation}

If $\lambda_{0} \neq 0$, by cancelling the poles in
\eqref{eq-sigma*eta}, we get
\[\tilde{\omega} = x \tilde{f}   \tilde{\eta} = \lambda_{0} \tilde{f} dx +
 x \sum_{i=1}^{r} \lambda_{i} \tilde{f}_{1} \cdots \widehat{ \tilde{f}_{i}} \cdots \tilde{f}_{r} d \tilde{f}_{i},\]
 where,
  as usual, ``\, $\widehat{\ \ }$\, '' indicates the absence of the corresponding factor.
It is clear that $D = \sigma^{-1}(0)$ is invariant by
$\tilde{\omega}$, that is, the blow-up $\sigma$ is non-dicritical.
Comparing with \eqref{eq-divided-blowup} and \eqref{eq-sigma*omega},
we conclude that $m_{0} = m-1$, and, thus, $\eta$ is $1$-faithful.

On the other hand, if $\lambda_{0} = 0$, then
\begin{equation}
\label{eq-tilde-omega}
\tilde{\omega} =  \tilde{f}   \tilde{\eta}=
 \sum_{i=1}^{r} \lambda_{i} \tilde{f}_{1} \cdots \widehat{ \tilde{f}_{i}} \cdots \tilde{f}_{r} d \tilde{f}_{i}.
 \end{equation}
If we suppose $\eta$ to be $1$-faithful then, by definition, $m-1 = m_{0} = \nu_{0}(f \eta)$.
Looking again at \eqref{eq-divided-blowup} and \eqref{eq-sigma*omega}, the equality
$m =  m_{0} + 1$
implies that $\sigma$ is dicritical. That is, $D$ is non-invariant,
which  is   equivalent to the non-vanishing of
the term independent of $x$ in the coefficient of $d t$ in   $\tilde{\omega}$.

Let us thus assume  $\eta$ to be $1$-faithful and main resonant, i.e. $\lambda_{0}=0$.
Our goal is to track the tangency points between $\tilde{\LL} = \sigma^{*} \LL$ and  $D$.
The term independent of $x$ in the coefficient of $d t$ in   $\tilde{\omega}$   is
\begin{multline}
\label{eq-tang-dt}
\ \ \ \sum_{i=1}^{r} \lambda_{i} \nu_{i} (t-\alpha_{1})^{\nu_{1}}
\cdots (t-\alpha_{i})^{\nu_{i}-1} \cdots (t-\alpha_{r})^{\nu_{r}}  \\
= (t-\alpha_{1})^{\nu_{1}-1} \cdots
(t-\alpha_{r})^{\nu_{r}-1}
\sum_{i=1}^{r} \lambda_{i} \nu_{i} (t-\alpha_{1})
\cdots \widehat{(t-\alpha_{i})} \cdots (t-\alpha_{r}) .
\end{multline}
Let $a_{1}, \ldots, a_{\ell}$ denote the distinct points of  the set $\{\alpha_{1},\ldots,\alpha_{r}\}$.
For $k=1,\ldots, \ell$, let $I_{k} = \{1 \leq i \leq r;\ \alpha_{i} = a_{k}\}$.
For each $k$, set
\[m_{k} = \sum_{i \in I_{k}} \nu_{i} \qquad \text{and} \qquad \rho_{k} =  \sum_{i \in I_{k}} \lambda_{i} \nu_{i} .\]

With this notation, we have that \eqref{eq-tang-dt} is   the polynomial
\begin{equation}
\label{eq-polynomial-pdot}
P_{\eta}(t) =
\sum_{k=1}^{\ell} \rho_{k} (t-a_{1}) \cdots \widehat{(t-a_{k})} \cdots (t-a_{\ell})
\end{equation}
times $(t-a_{1})^{m_{1}-1} \cdots
(t-a_{\ell})^{m_{\ell}-1}$.
The preceding discussion allows us to register  the following conclusion:
\begin{fact}
\label{fact-faithfull-P-eta}
A logarithmic $1$-form $\eta$, main resonant at $0 \in \C^{2}$,  is   $1$-faithful  if and only if  $P_{\eta} \neq 0$.
\end{fact}

The points of tangency between $\tilde{\LL}$  and $D$  out of the set $\{a_{1},\ldots,a_{\ell}\}$
are  the roots of $P_{\eta}(t)$.
At each one of them, $\tilde{\eta}$ is closed and holomorphic and, thus,  $\tilde{\LL}$ has a holomorphic first integral.
The roots of $P_{\eta}(t)$ determine points
 of two kinds, depending on $\tilde{\LL}$ being transverse to $D$ or not.
At a point $p \in D$ of the first kind,  the leaf through $p$ is contained in the singular set of $\tilde{\eta}$.
However, if $p \in D$ is of  the second kind, the reduction of singularities of $\LL$ will be non trivial over $p$, formed by
non-dicritical blow-ups, which give
 rise to some isolated separatrices for $\LL$. It is a escape point, with respect to the branches in the polar set of $\eta$
or with respect to any dicritical quadruplet $\QQ = (\Delta,\Sigma,\Upsilon,\Lambda)$ for which
$\eta$ is $\QQ$-logarithmic.
Observe that, in this latter case,   $\eta$ is faithful by hypothesis and,
from Assertion \ref{assertion-r2} in the proof of
Theorem \ref{teo-quadruplet},  we have that,  in \eqref{eq-polynomial-pdot},  $\rho_{k} \neq 0$, for  $k=1,\ldots, \ell$.

The vanishing of the weighted balance of residues, that is $\sum_{i=1}^{r} \nu_{i} \lambda_{i} = \sum_{k=1}^{\ell}   \rho_{k} = 0$,
gives that $P_{\eta}(t)$ has degree at most $\ell -2$. The coefficient of the term of degree $\ell -2$ is,
 up to changing sign,
\begin{equation}
\label{eq-tk-2}
 - \sum_{k=1}^{\ell}  \left( \sum_{1 \leq i \leq \ell, \, i \neq k }  \rho_{i} \right)a_{k}
=  \sum_{k=1}^{\ell}    \rho_{k}a_{k}.
\end{equation}
Equivalent calculations in the coordinates $(u,y)$ of the blow-up, where $x = uy$, lead to the polynomial
\[\hat{P}_{\eta}(u) = - \sum_{k=1}^{\ell} \rho_{k}  a_{k} (1-a_{1}u)
\cdots \widehat{(1-a_{k}u)} \cdots (1-a_{\ell}u),\]
whose constant term is $ -\sum_{k=1}^{\ell}    \rho_{k}a_{k}$.
Thus, $P_{\eta}(t)$ has degree $\ell -2$ if and only if  $\hat{P}_{\eta}(0) \neq 0$, which means
that  $\tilde{\LL}$  and $D$ are transverse at $(u,y) = (0,0)$.
Thus, by choosing appropriately the blow-up coordinates, we can always assume that $P_{\eta}(t)$ has degree  $\ell -2$ and, thus,
 the search for escape points of $\eta$ on $D$ can be entirely done in  the coordinates $(x,t)$. In particular, if $\ell = 2$, $P_{\eta}$ is constant
and there are no escape points.

Let us define
\begin{equation}
\label{eq-rational-eta}
R_{\eta}(t) =  \sum_{k=1}^{\ell} \frac{ \rho_{k}}{t-a_{k}} = \frac{P_{\eta}(t)}{Q_{\eta}(t)} ,
\end{equation}
where $Q_{\eta}(t) =  (t-a_{1}) \cdots (t-a_{\ell})$.  We name   $R_{\eta}$ \emph{escape function}. The affine zeroes of this rational function
are the points of tangency   considered, among which we find the  escape points of $\eta$ on $D$.
As a consequence of Fact \ref{fact-faithfull-P-eta}, we have:
\begin{fact}
\label{fact-faithfull-R-eta}
A logarithmic $1$-form $\eta$, main resonant at $0 \in \C^{2}$,  is   $1$-faithful  if and only if  $R_{\eta} \neq 0$.
\end{fact}

The above is the key to prove the following result, which was used in the dicritical amalgamation in the proof of
Theorem \ref{teo-quadruplet}:

\begin{prop}
\label{prop-faithfull-amalgamation}
Let
$\eta = \eta_{1} + \cdots + \eta_{\ell}$, where each
$\eta_{j}$ is a germ of logarithmic $1$-form at $(\C^2,0)$ and $\ell \geq 2$.
Suppose that:
\begin{itemize}
\item the tangent cone of the polar set of each $\eta_{k}$ is a singleton, say $p_{k} \in D = \sigma^{-1}(0)$;
\item $p_{1}, \ldots, p_{\ell} \in D$ are distinct points;
\item if $\rho_{k}$ denotes the weighted balance of residues of $\eta_{k}$, then
$\rho_{1} + \cdots +  \rho_{\ell} = 0$;
\item  $\rho_{k} \neq 0$ for some $k$.
\end{itemize}
Then,  $\eta$ is $1$-faithfull.
\end{prop}
\begin{proof} We take advantage of the above notation and write the  blow-up $\sigma$ at $0 \in \C^{2}$ in coordinates $(x,t)$ such that $\sigma(x,t) = (x,tx)$, in which
the points
$p_{k}$ have coordinates $(x,t) = (0,a_{k})$, for $k=1,\ldots,\ell$.
On account of Fact \ref{fact-faithfull-R-eta}, it is enough to see that $R_{\eta} \neq 0$.
This is however evident from formula \eqref{eq-rational-eta}, since some $\rho_{k}$ is non-zero.
\end{proof}

\subsection{Logarithmic modifications}
We remain in the setting of the previous subsection,
 keeping the notation.
Let $a_{1}', \ldots, a_{s}' \in \C$ be distinct numbers, none of them in $\{a_{1},\ldots,a_{\ell}\}$.
Let $\lambda_{1}', \ldots, \lambda_{s}' \in \C^{*}$. A \emph{$1$-logarithmic modification} of $\eta$ with
parameters $\tau = \{(a_{1}',\lambda_{1}'), \ldots, (a_{s}',\lambda_{s}')\}$ is a logarithmic $1$-form of the kind
\[  \eta^{\tau} = \eta + \lambda_{1}' df_{1}' /f_{1}' + \cdots + \lambda_{s}' df_{s}' /f_{s}',\]
where, for $i=1,\ldots,s$, the function $f_{i}' \in \OO_{0}$ is an equation  of a smooth branch  at $(\C^{2},0)$ whose
tangent cone is $t = a_{i}'$. If $\lambda_{1}' + \cdots + \lambda_{s}'= 0$
then the polynomial $P_{\eta^{\tau}}$ has degree at most $\ell + s - 2$.
In analogy with \eqref{eq-tk-2},
the coefficient of its term of degree  $\ell + s - 2$  is, up to  sign,
\begin{equation}
\label{eq-tk3}
\sum_{k=1}^{\ell}    \rho_{k}a_{k} +  \sum_{i=1}^{s}   \lambda_{i}'a_{i}'.
\end{equation}
Supposing that  $P_{\eta}$ has degree   $\ell   - 2$,
then,  for a generic choice of $(a_{1}',\ldots,a_{s}',\lambda_{1}',\ldots,\lambda_{s}')$ in the hyperplane of equation
$\lambda_{1}' + \cdots + \lambda_{s}' = 0$ in $\C^{2s}$,
 the degree of $P_{\eta^{\tau}}$ is precisely $\ell + s - 2$, so that the escape points of  $\eta^{\tau}$
are found among the affine roots of  $R_{\eta^{\tau}}$.  In this case, we say that  both
$\tau$  and the corresponding $\eta^{\tau}$ are \emph{balanced}.

Now, let $\QQ = (\Delta,\Sigma,\Upsilon,\Lambda)$ be a dicritical quadruplet at $(\C^{2},0)$, main resonant at
$0 \in \C^{2}$, and   $\pi$ be its underlying sequence of blow-ups. Suppose that
$\{a_{1},\ldots,a_{\ell}\}$   lists all points of the tangent cone of $\sep_{0}$
--- which coincides with the intersection of the support $\supp_{1}$ with  $D = \sigma^{-1}(0)$, where $\sigma$ is the initial
blow up of $\pi$ ---
as values of $t$ of
  coordinates $(x,t)$ for $\sigma$.
As in the above paragraph,
let $\tau$ be a balanced set of parameters, accompanied by a choice of smooth branches
 $S_{i}' \in \br_{0}$ of equations $f_{i}'=0$, for $i=1,\ldots,s$.
We define the \emph{1-logarithmic modification} of $\QQ$ with parameters
$\tau$ as the dicritical quadruplet
$\QQ^{\tau} = (\Delta^{\tau},\Sigma^{\tau},\Upsilon^{\tau},\Lambda^{\tau})$ obtained from $\QQ$ in the following manner:
\begin{itemize}
  \item $\Delta^{\tau} = \Delta$;
  \item $\Sigma^{\tau}$ is built upon the decomposition at level $0$  obtained  by attaching to $\sepdic_{0}$  the
  curves   $S_{i}'$, for $i=1,\ldots,s$.
  \item  $\Upsilon^{\tau}$ preserves the indices of  $\Upsilon$ and assigns  to the new elements of $\sepdic_{0}$ indices $I_{0}(S_{i}') =  1$.
  \item $\Lambda^{\tau}$ preserves the residues of $\Lambda$ and  assigns  residues
  $\lambda(S_{i}') = \lambda_{i}'$.
 \end{itemize}
In this framework, we have:

\begin{prop}
\label{prop-log-modification}
Suppose that $\eta$ is a $\QQ$-logarithmic $1$-form for some dicritical quadruplet $\QQ = (\Delta,\Sigma,\Upsilon,\Lambda)$
at $(\C^{2},0)$, main resonant at
$0 \in \C^{2}$. Let $\eta^{\tau}$ be a balanced $1$-logarithmic
modification of $\eta$ with parameters $\tau$.
Then $\eta^{\tau}$ is $\QQ^{\tau}$-logarithmic.
\end{prop}
\begin{proof} First,  the initial  blow-up $\sigma$  is dicritical for  $\eta^{\tau}$. Indeed,
$\eta^{\tau}$ is main resonant and $1$-faithful at $0 \in \C^{2}$, the latter being a consequence  of $R_{\eta^{\tau}} \neq 0$.
Next, we examine the points of the support of
the dicritical duplet $\DD^{\tau} = (\Delta^{\tau},\Sigma^{\tau})$ over $D = \sigma^{-1}(0)$. The points   $p_{k}$, corresponding to
$t= a_{k}$, for $k=1,\ldots,\ell$,
are simultaneously in the supports  of  $\DD^{\tau}$ and $\DD = (\Delta,\Sigma)$. At each of these points,
 $\sigma^{*}\eta^{\tau}$  differs from $\sigma^{*} \eta$
by a closed holomorphic $1$-form. Hence also $\sigma^{*}\eta^{\tau}$ is $\QQ_{p_{k}}$-logarithmic, which is the same of being $\QQ_{p_{k}}^{\tau}$-logarithmic,
since $\QQ_{p_{k}}^{\tau}$ and $\QQ_{p_{k}}$ coincide.   The other points of the support of $\DD^{\tau}$ are $p_{1}',\ldots,p_{s}' $, corresponding to
$t= a_{1}',\ldots,t=a_{s}'$. The fact that  $\sigma^{*}\eta^{\tau}$ at each $p_{i}'$ is  $\QQ_{p_{i}'}^{\tau}$-logarithmic is obvious.
\end{proof}

Suppose that  $\QQ = (\Delta,\Sigma,\Upsilon,\Lambda)$ has height $n = h(\Delta)$. Let
 $d$ be the number of dicritical blow-ups  of $\Delta$ and
denote by $q_{j_{k}} \in \D_{j_{k}}$ their centers, where $0 \leq j_{1} < \cdots < j_{d} < n$.
Starting with $q= q_{j_{d}}$, we can consider the localization $\QQ_{q}$ and, for a  balanced set of parameters  $\tau_{q}$, perform
 the corresponding $1$-logarithmic modification, obtaining a dicritical quadruplet
 $\QQ_{q}^{\tau_{q}}$ at $(\tilde{M}_{j_{d}},q)$ = $(\tilde{M}_{j_{d}},q_{j_{d}})$. The incorporation of new dicritical separatrices and their residues
at   levels of $\QQ$ lower than $j_{d}$ gives rise to a new dicritical quadruplet
 $\QQ^{\tau_{q}} = (\Delta^{\tau_{q}} ,\Sigma^{\tau_{q}},\Upsilon^{\tau_{q}},\Lambda^{\tau_{q}})$ at $(\C^{2},0)$.
 Since $\tau_{q}$ is balanced and all new
dicritical separatrices introduced are equisingular, their weighted balances of residues, at levels lower than $j_{d}$, will always vanish.
Therefore, at any level, these new separatrices give no contribution to the weighted balances of residues prescribed by $\QQ$.
Descending the dicritical structure, we  carry out successive $1$-logarithmic modifications at each $q_{j_{k}}$ with
balanced set of parameters $\tau_{j_{k}}$.
The outcome is a dicritical quadruplet
$\QQ^{\textsc{t}} = (\Delta^{\textsc{t}},\Sigma^{\textsc{t}},\Upsilon^{\textsc{t}},\Lambda^{\textsc{t}})$, where
the superscript $\textsc{T}$ makes reference to the set of information given by the pairs $(q_{j_{k}},\tau_{j_{k}})$,
 for $k=1,\ldots,d$, along with the new separatrices introduced.
Now, if $\eta$ is a $\QQ$-logarithmic $1$-form at $(\C^{2},0)$, the successive $1$-logarithmic modifications just described
can be applied to $\eta$. The result is a germ logarithmic $1$-form $\eta^{\textsc{t}}$, which, by a careful
application of Proposition \ref{prop-log-modification}, can be shown to be $\QQ^{\textsc{t}}$-logarithmic.
It is noteworthy that a $1$-logarithmic modification at a point $q_{j_{k}}$ modifies
the escape function  on $D_{j_{k}+1}=\sigma_{j_{k} +1}^{-1}(q_{j_{k}})$ but does not change escape functions on dicritical components
at higher or lower  levels.

\subsection{Symmetric logarithmic modifications}

Suppose now that $\eta$ is a symmetric logarithmic $1$-form at $(\C^{2},0)$, defining a germ of singular holomorphic foliation $\LL$.
A set of parameters $\tau = \{(a_{1}',\lambda_{1}'), \ldots, (a_{s}',\lambda_{s}')\}$ is \emph{symmetric} if its invariant by conjugation.
In this case
we say that $\tau$ parametrizes a
 \emph{symmetric $1$-logarithmic modification} of $\eta$ provided the set of equations
$f_{1}', \ldots, f_{s}' \in \OO_{0}$ is also invariant by the involution $J$.
In particular, if   $a_{i}' \in \R$, then also  $\lambda_{i}' \in \R$ and  $f_{i}' \in \OO_{0}$
defines a real curve. As a consequence, $\eta^{\tau}$   is also a symmetric logarithmic $1$-form.
Remark that, if $\eta$ is main resonant  and symmetric, then the set of parameters involved in formula \eqref{eq-rational-eta} is symmetric under conjugation. Thus,
$P_{\eta}$  can be obtained with real coefficients (the same for $P_{\eta^{\tau}}$, for a balanced symmetric $1$-logarithmic modification $\eta^{\tau}$).
In particular, the number of real roots of   $P_{\eta}$ and its degree have the same parity.

Balanced symmetric  1-logarithmic modifications can be used to  eliminate real  escape points  of symmetric logarithmic $1$-forms, as
described in the next proposition.
\begin{prop}
\label{prof-no-real-roots}
 Suppose that $\eta$ is a main resonant $1$-faithful symmetric logarithmic 1-form at $(\C^{2},0)$.
Then there is a balanced  symmetric set of real parameters $\tau$ such that
the corresponding symmetric
 $1$-logarithmic modification  $\eta^{\tau}$  has no  escape points in the   real trace of the divisor of the quadratic blow-up
at $0 \in \C^{2}$.
\end{prop}
\begin{proof}
As above, we fix $(x,t)$ coordinates for the blow-up $\sigma$ at $0 \in \C^{2}$ so that
 the affine zeroes of the escape function $R_{\eta}(t)$ give all tangency points between $\tilde{\eta} = \sigma^{*}\eta$ and $D=\sigma^{-1}(0)$.
We can suppose that $P_{\eta}(t)$ has even degree. If this is not so,
we make a balanced symmetric 1-logarithmic modification with parameters $\{(a_{1}',\lambda_{1}'), (a_{2}',\lambda_{2}'), (a_{3}',\lambda_{3}')\}$, where $a_{i}',\lambda_{i}' \in \R^{*}$ for $i=1,2,3$, with $\lambda_{1}',\lambda_{2}',\lambda_{3}'$ sufficiently small. 
Let $t_{0}$ be a real zero of $R_{\eta}(t)$, which can be supposed to be $0 \in \C$ after a real translation in the coordinate $t$.
We  write, locally,
\[ R_{\eta}(t) = \frac{P_{\eta}(t)}{Q_{\eta}(t)}=
t^{n}g(t),\]
for some $n>0$, where $g(t)$ is a  holomorphic function in
a neighborhood of $0 \in \C$ such that $\alpha = g(0) \in \R^{*}$.
Choose $r>0$ sufficiently small such that
 $ |\alpha|/2 < |g(t)| < 2|\alpha|$
over $\overline{ \mathbb{D}}_{r}$.
This implies that
$g(t)$ has no zeroes over $\overline{\mathbb{D}}_{r}$ and,
 for  $t \in \R$  with $|t|<r$, $g(t)$ and $\alpha = g(0)$ have
the same sign.
Note also that  $|t^{n} g(t)| > r^{n}|\alpha|/2$ over $\partial \overline{ \mathbb{D}}_{r}$.

 \smallskip \smallskip
\par \noindent
\underline{Case 1}: $n$ even.
We perform a  balanced 1-logarithmic modification with parameters $\tau_{0} = \{(a,\lambda),(-a,-\lambda)\}$, where the sufficiently small $a,\lambda \in \R^{*}$, with $a>0$, are   to be specified later,  getting
a logarithmic $1$-form $\eta^{\tau_{0}}$ for which
\[R_{\eta^{\tau_{0}}}(t) = t^{n} g(t) + \frac{\lambda}{t-a} + \frac{-\lambda}{t+a} =
 \frac{t^{n}(t^{2}-a^{2}) g(t) + 2a\lambda}{t^{2}-a^{2}}.\]
Choose $|a| < r/2$, so  that $|\lambda/(t \pm a)| < 2|\lambda|/r$ 
over  $\partial \overline{ \mathbb{D}}_{r}$.
If we also choose $\lambda \in \R^{*}$ such that $|\lambda| <  r^{n+1} |\alpha|/8$, we assure that
 \[ \left|\frac{\lambda}{t-a} + \frac{-\lambda}{t+a} \right| <
 \frac{4 |\lambda|}{r} <
 \frac{ r^{n} |\alpha|}{2} < |t^{n} g(t)| \]
  over $\partial \overline{\mathbb{D}}_{r}$. Hence, applying  Rouch\'es's theorem,
 $ t^{n} g(t) + \lambda/(t-a) - \lambda/(t+a)$  has exactly $n+2$ zeroes in $\mathbb{D}_{r}$.
 They  are roots of the function  $t^{n}(t^{2}-a^{2}) g(t) + 2a\lambda$.
If  $t_{0} \in \R$ is one of these roots, then
 \begin{equation}
 \label{eq-real-root0}
 t_{0}^{n}(t_{0}^{2}-a^{2})  = - \frac{2a\lambda}{g(t_{0})}.
 \end{equation}
 Observe that
$t^{n}(t^{2}-a^{2}) \geq -2 K a^{n+2}$ for every $t \in \R$,
where $K = \frac{1}{n+2}(\frac{n}{n+2})^{\frac{n}{2}}$.
Choose first a sufficiently small  $a \in \R^{*}$, with $a>0$, and then pick $\lambda \in \R^{*}$, also small,  but satisfying $|\lambda| >  2 |\alpha|K  a^{n+1} $.
We then have
 \[\frac{2a|\lambda|}{|g(t_{0})|} > ( 2 a ) 2 |\alpha| K  a^{n+1}  \frac{1}{2|\alpha|} = 2   K a^{n+2}.\]
Taking  $\lambda$ with the same sign as $\alpha = g(0)$,  we have
\[-\frac{2a\lambda}{g(t_{0})} < - 2   K a^{n+2}  \leq t^{n}(t^{2}-a^{2}) \ \ \ \text{for every}\ t \in \R.\]
Comparing with   \eqref{eq-real-root0}, we conclude that the real root $t_{0}$ cannot exist.

\smallskip \smallskip
\par \noindent
\underline{Case 2}: $n$ odd.
Now we do a non-balanced 1-logarithmic modification with parameter $\tau_{0} = \{(a,\lambda)\}$,  with sufficiently small $a,\lambda \in \R^{*}$.   The
  logarithmic $1$-form $\eta^{\tau_{0}}$ obtained is such that
\[R_{\eta^{\tau_{0}}}(t) =
 t^{n} g(t) + \frac{\lambda}{t-a} =
 \frac{t^{n}(t-a) g(t) + \lambda}{t-a}.\]
If  $|a| < r/2$, then $|\lambda/(t-a)| < 2|\lambda|/r$ over  $\partial \overline{ \mathbb{D}}_{r}$.
 Choosing $\lambda \in \R^{*}$ such that $|\lambda| <  r^{n+1} |\alpha|/4$, we assure that
 $|\lambda/(t-a)| < r^{n} |\alpha|/2 < |t^{n} g(t)|$ over $\partial \overline{\mathbb{D}}_{r}$. Hence, applying again Rouch\'es's theorem,
 $ t^{n} g(t) + \lambda/(t-a)$  has exactly $n+1$ zeroes in $\mathbb{D}_{r}$.
 These zeroes  are roots of the function  $t^{n}(t-a) g(t) + \lambda$.
If  $t_{0} \in \R$ is one of these roots, then
 \begin{equation}
 \label{eq-real-root1}
 t_{0}^{n}(t_{0}-a)  = - \frac{\lambda}{g(t_{0})}.
 \end{equation}
We have that
$t^{n}(t-a) \geq - K a^{n+1}$ for every $t \in \R$,
where $K =  \frac{1}{n+1} ( \frac{n}{n+1} )^{n}$. 
Choose first a sufficiently small  $a \in \R^{*}$ and then pick $\lambda \in \R^{*}$, also small,  but satisfying $|\lambda| >  2 |\alpha|K  a^{n+1} $.
We then have
 \[\frac{|\lambda|}{|g(t_{0})|} >  2 |\alpha|K  a^{n+1} \frac{1}{2|\alpha|} = K a^{n+1}.\]
Taking  $\lambda$ with the same sign as $\alpha = g(0)$,  we have
\[-\frac{\lambda}{g(t_{0})} < - K a^{n+1}  \leq t^{n}(t-a) \ \ \ \text{for every}\ t \in \R.\]
Comparing with   \eqref{eq-real-root1}, this  precludes the existence of such a real root.

 \smallskip \smallskip
\par
We complete the proof by performing successive individual 1-logarithmic modifications as above, whenever we find
a real root for the escape function.
Once a root has been worked out, generating new roots in a small disc around it,  subsequent 1-logarithmic
modifications
will not perturb any of these new roots back to the real axis,
 provided all  parameters are sufficiently small.
Finally, since $P_{\eta}(t)$ has real coefficients and even degree, there is an even  number of real roots with $n$ odd. Thus, if any such a root exists,
there are at least two of them, allowing us to choose the residues $\lambda$  in a way that
the overall 1-logarithmic modification is balanced.
\end{proof}

Remark that in the proof of Proposition \ref{prof-no-real-roots}, all residues can be obtained in $\Q^{*}$. This is particularly important
if our goal is to use logarithmic models in order to produce real meromorphic functions, as in Theorem
\ref{teo-sectorial-modeled-meromorphic} below.
The iterated application of Proposition \ref{prof-no-real-roots} in a symmetric dicritical structure
gives promptly the following result, which was used in Theorem \ref{teo-strict-logmodel-real} in order to
obtain a strict real logarithmic model.
\begin{prop}
\label{prop-withour-real-escape}
Let $\QQ = (\Delta,\Sigma,\Upsilon,\Lambda)$ be a symmetric dicritical quadruplet at $(\C^{2},0)$
and $\eta$ be a symmetric $\QQ$-logarithmic $1$-form. Then there is a finite set of balanced  symmetric logarithmic modifications, with parameters  $\textsc{T}$,
such that $\eta^{\textsc{t}}$ is a symmetric $\QQ^{\textsc{t}}$-logarithmic $1$-form without real escape points.
Besides, all individual $1$-logarithmic modifications can be taken with real parameters, with residues in $\Q^{*}$.
\end{prop}

Proposition \ref{prop-withour-real-escape} completes the proof of Theorem \ref{teo-strict-logmodel-real}
on the existence of strict logarithmic models for germs of singular real analytic foliations.
Indeed, having a germ of real analytic foliation $\F_{\R}$ of real generalized curve type, we provide a q-dicritical triplet
$\TT^{\aleph}_{\R}  = \TT^{\aleph}_{\R}(\F_{\R})$ and Theorem \ref{teo-real-logmodel} gives it a logarithmic model $\eta_{\R}$.
  In this process, $\TT^{\aleph}_{\R} $ is complexified, then completed into a symmetric
 dicritical triplet $\TT$ with real indices, which turns into
a symmetric dicritical quadruplet $\QQ = (\TT,\Lambda)$  with real residues,
to finally obtain a $\QQ$-logarithmic $1$-form $\eta$, whose decomplexification is
$\eta_{\R}$.
If  $\eta_{\R}$ fails to be a strict  logarithmic model for $\F_{\R}$, that is to say, if   escape points exist,
  we   apply Proposition \ref{prop-withour-real-escape} to $\eta$.
In this case, the set of parameters  $\textsc{T}$ will be chosen in such a way  that all separatrices involved are actual real dicritical separatrices of $\F$, the
germ of singular holomorphic foliation defined by $\eta$. The corresponding logarithmic modifications  then produce  a new q-dicritical triplet $(\TT^{\aleph}_{\R})^{\textsc{t}}$, the real trace of $\TT^{\textsc{t}}$,
which in practice only incorporates the data of the real separatrices introduced,
being also a q-dicritical triplet associated with $\F_{\R}$.
Thus, $\eta_{\R}^{\textsc{t}}$, the decomplexification of $\eta^{\textsc{t}}$, is a strict logarithmic model for $\F_{\R}$.


\section{Bendixson's sectorial decomposition}
\label{section-sectorial}

The existence of a sectorial decomposition for a planar   real analytic vector field with isolated singularity
appeared in the seminal  paper  of I. Bendixson  \cite{bendixson1901}, where, following the ideas of H. Poincar\'e, he developed a qualitative study of
the orbits of a planar vector field.
Bendixson described  and proved the finiteness of such a decomposition
in a sufficiently small neighborhood of the singularity, say $0 \in \R^{2}$,
when some orbit approaches the singularity  with a ``determined tangent''.
In modern terminology, this is a \emph{characteristic orbit} and the situation considered is called \emph{non-monodromic}.
In this case, it turns out that all orbits approaching the origin have tangents, defined as limits of secants.
Finitely many of these orbits will be the  boundary of sectors,
which have a topological classification as \emph{hyperbolic}, \emph{parabolic}  or \emph{elliptic}.
In a hyperbolic sector no trajectory has the origin in its limit set,
in a parabolic one all trajectories  have the origin either in the $\alpha$
or in the $\omega$-limit sets and, finally,
in an elliptic sector all trajectories  have the origin both in their $\alpha$ and  $\omega$-limit sets.
We refer the reader to \cite{ilyashenko2008} for a proof of the existence of a sectorial decomposition based on the technique of reduction of singularities,
which unveils other  analytic aspects that will be relevant  to the forthcoming discussion.

Let $\pi_{\R}: (\tilde{M}_{\R}, \D_{\R}) \to (\R^{2},0)$ be a   sequence of real quadratic blow-ups and let
$\rho_{\R}: (\tilde{N}_{\R}, \E_{\R}) \to (\R^{2},0)$ be the corresponding   sequence of
trigonometric blow-ups, where $\tilde{N}_{\R}$ is a real analytic surface with boundaries and corners (see definitions in \cite{ilyashenko2008}).
Sequences of trigonometric blow-ups to be considered henceforth are of this kind.
There is a  canonical analytic map $\psi:(\tilde{N}_{\R}, \E_{\R}) \to (\tilde{M}_{\R}, \D_{\R})$
satisfying $\pi_{\R} \circ \psi =  \rho_{\R}$, obtained by
the consideration of the double covering $\mathbb{S}^{1} \to \mathbb{P}^{1}_{\R}$
for each single blow-up.
We associate with each component $D \subset \D_{\R}$ its real \emph{infinitesimal class} $\kappa_{D}$.
If $D$ has valence $m$, then  the subset of trace points of $D$
 has $m$ connected components, whose
 pre-images by $\psi$   are $2m$ connected components of the regular
part of $\E_{\R}$. They determine $2m$ distinct  \emph{infinitesimal semiclasses} and all these
semiclasses are said to be
\emph{correlate} to the infinitesimal class $\kappa_D$  and  also \emph{correlate} to each other, and the
same is said of the corresponding irreducible components of $\E_{\R}$ and $\D_{\R}$.
The infinitesimal semiclass of $E \subset \E_{\R}$ will be denoted by $\hat{\kappa}_E$.
Let $E_{x} \subset \E$ be the component corresponding to the positive $x$-semiaxis.
If $E, E' \subset \E$ are two distinct components, the  notation $\hat{\kappa}_{E} \prec \hat{\kappa}_{E'}$
 means that either $E = E_{x}$ or $E$ lies between $E_{x}$ and $E'$, considering the counterclockwise
 orientation around $0 \in \R^{2}$.

Let $U$ be a small neighborhood of $0 \in \R^{2}$.
An analytic \emph{semicurve} in $U \setminus \{0\}$ is   defined by a real analytic parametrization  of the kind  $\gamma:[t_{0},\infty) \to U \setminus \{0\}$ or   $\gamma:(-\infty,t_{0}] \to U \setminus \{0\}$,
for some $t_{0} \in \R$.
We shall only consider analytic semicurves $\gamma$  in $U \setminus \{0\}$ whose limit set is either $\ell(\gamma) = \{0\}$ or
$\ell(\gamma) = \emptyset$.
Let $X_{\R}$ be a germ of real analytic vector field, with isolated singularity at $(\R^{2},0)$, of non-monodromic type.
Any integral semicurve  $\gamma$   of $X_{\R}$ such that $\ell(\gamma)=0$  has the
\emph{iterated tangents} property, meaning that $\gamma$, as well as all its lifts by successive blow-ups, have   tangents at their limit points
(we refer to  \cite[Sec. 2.2]{sanz2000} for a proof).
In a sufficiently small neighborhood of $0 \in \R^{2}$, an integral semicurve does not intercept a fixed real analytic  branch (unless the semicurve is contained in the branch). Also, two integral semicurves do not intercept.
Thus, a pair   $\gamma,\gamma'$  of integral solutions of $X_{\R}$, both accumulating at the origin, can be counterclockwise ordered
having the positive $x$-semiaxis $\gamma_{x}$ as reference:
$\gamma \prec \gamma'$ means that either  $\gamma = \gamma_{x}$  or $\gamma$ lies between $\gamma_{x}$ and $\gamma'$ in a sufficiently small neighborhood of $0 \in \R^{2}$.
We also say that   $\gamma$ has type $\hat{\kappa}_{E}$, for some $E \subset \E$, if the $\rho_{\R}$-lifiting of $\gamma$ touches
$E$ in a trace point. In this case, if $E',E'' \subset \E$ are such that $\hat{\kappa}_{E'} \prec \hat{\kappa}_{E} \prec  \hat{\kappa}_{E''}$, we denote
$\hat{\kappa}_{E'} \prec \gamma \prec  \hat{\kappa}_{E''}$.

A \emph{sectorial  model}
corresponds to the prescription,
in a sufficiently small neighborhood $U$ of $0 \in \R^{2}$, of the following information:
\begin{itemize}
\item a finite number of real analytic  semicurves  $\gamma_{1},\ldots,\gamma_{n}$ in $U \setminus \{0\}$, with limit set $0 \in \R^{2}$, not intercepting each other and having the property of iterated tangents,
 such that $\gamma_{1} \prec \ldots \prec \gamma_{n} $;
\item a classification of the region $\chi = \chi(\gamma_{i},\gamma_{i+1})$, named \emph{sector}, formed by points of $U$   between $\gamma_{i}$ and $\gamma_{i+1}$, where
$i=1,\ldots,n$ and $\gamma_{n+1}  = \gamma_{1}$,
as hyperbolic, parabolic or elliptic,
with no parabolic sector neighboring other parabolic  sectors.
\end{itemize}
Evidently, in the first of these items, we could also ask other necessary conditions for the semicurves $\gamma_{i}$ to be integral curves of
a non-monodromic real analytic vector field.
Taking into account the machinery developed throughout this text, our proposal   is to produce examples of real analytic
vector fields whose sectorial decomposition fits a  prescribed sectorial model. Our techniques allow  us to do this
 in  a variety of  situations, having as output a vector field
of logarithmic nature,
i.e. whose dual $1$-form  is a multiple of a real logarithmic $1$-form by a unit in $C^{\omega}_{0}$.
We will exemplify this procedure within a  specific category of
non-monodromic real analytic vector fields that we call  \emph{$\ell$-analytic}, delimited in the following paragraph.

Denote the set of irreducible germs of real analytic curves at $(\R^{2},0)$ by $\br^{\R}_{0}$. If $S \in \br^{\R}_{0}$, then $S \setminus \{0\}$ splits into
two   \emph{semibranches},  $\gamma$ and $\gamma^{*}$, said to be \emph{adjoint} to each other.
Evidently, a semibranch $\gamma$ is an integral curve of a germ of real analytic vector field $X_{\R}$ if and only if $\gamma^{*}$ is.
An analytic semicurve $\gamma$ in $U \setminus \{0\}$   is analytic in the limit,
or   \emph{$\ell$-analytic} for short, if either
$\ell(\gamma) = \emptyset$ or
  $\ell(\gamma) = \{0\}$ and
there is  $S \in \br^{\R}_{0}$ such that $\gamma(t) \in S$ whenever $|t|$ is sufficiently large. In an abuse of   language,
we   say that $\gamma$ itself is an analytic semibranch. An analytic curve $\Gamma:(-\infty,\infty) \to  U \setminus \{0\}$ is $\ell$-analytic
if it can be partitioned in two analytic semi-curves, both $\ell$-analytic.

\begin{ddef}
\label{def-l-analytic}
{\rm We say that a germ of real analytic vector field $X_{\R}$ is \emph{$\ell$-analytic} if
all its orbits in $U \setminus 0$   are $\ell$-analytic,  for some small neighborhood $U$ of the origin.
}\end{ddef}

Clearly,
a vector field whose underlying singular  foliation is defined by the levels of a real meromorphic function
is $\ell$-analytic.
Simple
saddle-type   vector fields
--- whose  linear part has eigenvalues  $\lambda_{1}, \lambda_{2}$ satisfying  $\lambda_{2}/ \lambda_{1} \in \R_{-}$  ---
 are $\ell$-analytic. On the other hand, simple
node-type  vector fields
--- with eigenvalues
satisfying  $\lambda_{2} / \lambda_{1} \in \R_{+} \setminus \Q_{+}$  ---
 are not $\ell$-analytic.
A simple vector field with one zero eigenvalue --- an \emph{algebraic saddle-node} --- is $\ell$-analytic if and only if it is a topological saddle.
In particular, its   weak separatrix --- the one corresponding to the zero eigenvalue --- is convergent.
Indeed, otherwise, by the center manifold theorem, there would exist an orbit asymptotic to the non-convergent weak separatrix,
which would not be $\ell$-analytic. A monodromic germ of vector field is $\ell$-analytic if and only if it is of center type.
The concept of $\ell$-analytic vector field is evidently invariant under quadratic blow-ups.
Thus, if $\pi_{\R}: (\tilde{M}_{\R}, \D_{\R}) \to (\R^{2},0)$ is a reduction of singularities of $\F_{\R}$ by a   sequence of real quadratic blow-ups,
then $X_{\R}$ is $\ell$-analytic if and only if all simple  singularities  of $\tilde{\F}_{\R}= \pi_{\R}^{*}\F_{\R}$ over $\D_{\R}$  are topological saddles.

\begin{figure}[t]
\begin{center}
\includegraphics[width=5.5cm, angle=-55]{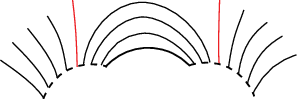}
\includegraphics[width=5.5cm, angle=-55]{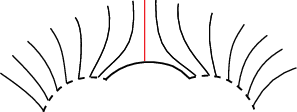}
\includegraphics[width=5.5cm, angle=-55]{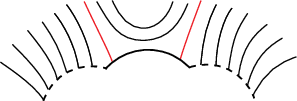}
\put(-390,30){$\underline{\tau = 0}$}
\put(-342,0){\scriptsize {$\xi'$}}
\put(-298,-62){\scriptsize {$\xi$}}
\put(-350,-75){\scriptsize {$p$}}
\put(-375,-45){\scriptsize {$p'$}}
\put(-343,-90){\scriptsize {$E$}}
\put(-387,-33){\scriptsize {$E'$}}
\put(-355,-55){\scriptsize {$\mathcal{A}$}}
\put(-313,-25){\scriptsize {$\chi(\xi,\xi')$}}
\put(-315,-37){\scriptsize {(elliptic)}}
\put(-290,-85){\scriptsize {$\chi_{E}$}}
\put(-375,17){\scriptsize {$\chi_{E'}$}}
\put(-252,30){$\underline{\tau = 1}$}
\put(-180,-26){\scriptsize {$\gamma$}}
\put(-212,-75){\scriptsize {$p$}}
\put(-237,-45){\scriptsize {$p'$}}
\put(-205,-90){\scriptsize {$E$}}
\put(-250,-33){\scriptsize {$E'$}}
\put(-220,-55){\scriptsize {$\mathcal{A}$}}
\put(-147,-90){\scriptsize {$\chi_{E}$}}
\put(-227,15){\scriptsize {$\chi_{E'}$}}
\put(-114,30){$\underline{\tau = 2}$}
\put(-64,-3){\scriptsize {$\gamma'$}}
\put(-26,-56){\scriptsize {$\gamma$}}
\put(-71,-75){\scriptsize {$p$}}
\put(-96,-45){\scriptsize {$p'$}}
\put(-64,-90){\scriptsize {$E$}}
\put(-110,-33){\scriptsize {$E'$}}
\put(-80,-55){\scriptsize {$\mathcal{A}$}}
\put(-32,-25){\scriptsize {$\chi(\gamma,\gamma')$}}
\put(-34,-37){\scriptsize {(hyperbolic)}}
\put(-9,-90){\scriptsize {$\chi_{E}$}}
\put(-89,15){\scriptsize {$\chi_{E'}$}}
\end{center}
\end{figure}

We will  try to present a more  refined description of the    sectorial  decomposition of an  $\ell$-analytic  vector field $X_{\R}$.
Let $\rho_{\R}: (\tilde{N}_{\R}, \E) \to (\R^{2},0)$ be the reduction of singularities of the associated singular foliation $\F_{\R}$  by a sequence of
trigonometric blow-ups and denote  $\tilde{\F}_{\R}^{\rho}= \rho_{\R}^{*}\F_{\R}$.
Let $E,E' \subset \E$  be consecutive  dicritical components with $\hat{\kappa}_{E} \prec \hat{\kappa}_{E'}$.
Let $\A$
denote
the union of invariant components of $\E$ between
$E$ and $E'$,  which is non-empty by the properties of the reduction of singularities.
Let $\tau$ denote  the number of separatrices --- i.e. invariant semibranches --- of  $\tilde{\F}_{\R}^{\rho}$ touching $\A$.
The figure gives a portrait for the cases $\tau=0,1$ and $2$.
Define $p = \A \cap E$ and $p' = \A  \cap E'$.
The dicritical components $E$ and $E'$
determine
two pieces of parabolic sectors,  $\chi_{E}$ and $\chi_{E'}$ --- here we say ``piece'' since they may be  proper subsectors of parabolic sectors.
We attach to them the infinitesimal semiclasses  $\hat{\kappa}_{E}$ and  $\hat{\kappa}_{E'}$, on the grounds that their interior orbits
are of types, respectively, $E$ and $E'$.
 Suppose   that $\tau = 0$.
Then, leaves of $\tilde{\F}_{\R}^{\rho}$ touching $E $ in points near  $p$
reach $E'$ in points near $p'$. Choosing $\xi$ and $\xi'$
orbits whose $\rho_{\R}$-lifts
touch, respectively, $E$ in a point sufficiently near $p$ and
$E'$ in a point sufficiently near $p'$, then $\chi(\xi,\xi')$ is an elliptic
sector. The semibranch $\xi$ will be the upper boundary of $\chi_{E}$ and
$\xi'$ will be the lower boundary of $\chi_{E'}$.
On the other hand,
 if  $\tau > 0$, then the first of the separatrices reaching   $\A$, say   $\gamma$,
will be the upper boundary of $\chi_{E}$ and the last one, say $\gamma'$,
will be the lower boundary of $\chi_{E'}$. If $\tau =1$, then $\gamma = \gamma'$ and  $\chi_{E}$
neighbor $\chi_{E'}$.  Hence, these pieces of parabolic sectors are merged, as part of the formation of a maximal parabolic sector.
On the other hand, if $\tau > 1$, there are $\tau-1$ hyperbolic sectors between  $\chi_{E}$ and $\chi_{E'}$, each
one bounded by a consecutive pair of separatrices of $\A$.
In other words, hyperbolic sectors are bounded by isolated invariant semibranches of $\F_{\R}$.
Observe that elliptic sectors always neighbor parabolic sectors and
their boundaries are always dicritical invariant semibranches of $\F_{\R}$.
By possibly reducing $U$ and changing the choice of some of these semibranches, we can also suppose
that a semibranch $\gamma$ is the boundary
 of a sector  or of a piece of sector if and only if $\gamma^{*}$ is.
Finally, it is also clear that if $E \subset \E$ is a dicritical component defining a  piece of parabolic sector, then all components of $\E$ correlate to $E$ also define pieces of parabolic sectors.

The above description matches the   model we describe next.
An \emph{infinitesimal sentence} is an object $\W = W_{1} W_{2} \cdots W_{r}$,
formed by \emph{words}  $W_{i}$   of one of two \emph{types}:  either  $W_{i}=\gamma_{i}$
is a semibranch
or $W_{i}= \hat{\kappa}_{i}$ is an infinitesimal semiclass.
 The following syntax must be obeyed:
\begin{itemize}
\item $W_{1} \prec W_{2} \prec  \cdots  \prec W_{r}$, where ``$\prec$'' denotes the counterclockwise order
of semibranches and semiclasses described above;
\item if $W_{i}= \gamma_{i}$ is a semibranch, then the adjoint $\gamma_{i}^{*}$ is also a word of $\W$;
\item if $W_{i}= \hat{\kappa}_{i}$ is an infinitesimal semiclass, then all semiclasses correlate to $\hat{\kappa}_{i}$
are words of $\W$.
\end{itemize}
For convenience, consider  $W_{0} = W_{r}$  and $W_{r+1} = W_{1}$. We then establish the following:
\begin{ddef}
\label{def-sectorial-modeled}
{\rm
 Let $\W$ be an infinitesimal sentence as above.
An \emph{$\ell$-analytic sectorial   model}
of pattern $\W$ is the prescription of a sectorial decomposition of a small neighborhood of $0 \in \R^{2}$
complying with the following conditions:
\begin{itemize}
\item if $W_{i} = \gamma_{i}$ is a semibranch, then it is the boundary of a hyperbolic sector or of a piece of parabolic  sector;
\item if $W_{i} = \hat{\kappa}_{i}$ is an infinitesimal semiclass, then it defines a  piece of parabolic sector;
if $W_{i-1}$ or $W_{i+1}$ is a semibranch, say $\gamma$, then $\gamma$ is in its  boundary;
\item if $W_{i} = \gamma$ and $W_{i+1} = \gamma'$ are semibranches, then $\chi(\gamma,\gamma')$ is a hyperbolic sector;
\item if $W_{i} = \hat{\kappa}$ and $W_{i+1} = \hat{\kappa}'$ are infinitesimal semiclasses, then there is an elliptic sector $\chi(\xi,\xi')$,
bounded by  semibranches $\xi$
and $\xi'$ of types, respectively,
$\hat{\kappa}$ and $\hat{\kappa}'$;
these semibranches are also boundaries of  pieces of parabolic sectors; 
\item   neighboring pieces of parabolic sectors are merged into a single parabolic sector;
\end{itemize}
}\end{ddef}
From the definition, we infer  that  parabolic sectors correspond, in the sentence $\W$, to    maximal sequences of words of alternating types.
Aiming at giving them a more accurate description, we could say, for instance, that each such a sequence   determines an \emph{infinitesimal multitype}
for the corresponding parabolic sector.

From our discussion, every $\ell$-analytic   vector field has a sectorial decomposition of the above type.
Reciprocally, we have:

\begin{maintheorem}
\label{teo-sectorial-modeled}
There are $\ell$-analytic   vector fields fitting any preassigned $\ell$-analytic  sectorial  model.
\end{maintheorem}
\begin{proof}
Let $\W$ be the infinitesimal sentence subjacent to the given $\ell$-analytic  sectorial  model.
Let $\sep \subset \br_{0}^{\R}$ be the set of branches whose semibranches are words of $\W$
and $\varepsilon = \varepsilon(\sep)$ be its real equisingularity class.
Let $\kappa$ be the real infinitesimal class that aggregates all semiclasses that are words of $\W$.
Let $\pi_{\R}$ be a sequence of real quadratic blow-ups that is a minimal simultaneous  realization for $\kappa$ and $\varepsilon$.
Consider   a real q-dicritical triplet $\TT^{\aleph}_{\R}$ built upon  $\pi_{\R}$, with the dicritical structure defined by $\kappa$ and separatrices given
by   $\sep$. As for the indices,  for each   singular point $p$ at  the final level, choose, for
the two local branches of the support at $p$, negative indices respecting (I.4).
By Theorem \ref{teo-real-logmodel},  there is a real  logarithmic 1-form  $\eta_{\R}$, with real residues,
which is a   real  logarithmic model for  $\TT^{\aleph}_{\R}$.
Using  Proposition \ref{prop-withour-real-escape}, after a finite number of logarithmic modifications, we obtain a real logarithmic $1$-form $\eta^{\textsc{t}}$ without
escape points.
We cancel the poles of $\eta^{\textsc{t}}$ and take   the dual vector field.
This vector field has  the  preassigned $\ell$-analytic  sectorial  model
\end{proof}

We can also obtain a  version of the above theorem concerning the existence of real meromorphic functions:
\begin{maintheorem}
\label{teo-sectorial-modeled-meromorphic}
There are real analytic meromorphic functions fitting any preassigned $\ell$-analytic  sectorial  model.
\end{maintheorem}
\begin{proof}
We follow the steps of the proof of Theorem \ref{teo-sectorial-modeled}, associating with the given sectorial model a dicritical class $\kappa$ and an equisingularity class
$\varepsilon$.
Then, as  in Theorem  \ref{teo-structure-meromporphic}, we produce a real meromorphic function upon these data,
using, in this process, logarithmic modifications with residues in $\Q^{*}$ in order to get rid of all real escape points, as proposed in
 Proposition \ref{prop-withour-real-escape}.
\end{proof}

\section{Examples}
\label{sec-examples}
As an illustration of methods and concepts introduced in the text, we present below a pair of examples.
In order to lighten notation, we denote by the same symbol a curve and its strict transform by blow-up maps.
Ambiguity will be avoided by making precise the level we are working at.
Also,  all  pictures presented correspond to the real trace.

\begin{example}{\rm
Consider, at $(\C^{2},0)$, the branches
\[ S_{1}: f_{1} = x-y^{2} = 0 \quad \text{and} \quad S_{2}: f_{2} = x+y^{2} = 0.\] 
Let $\pi = \sigma_{1} \circ \sigma_{2}$ be the sequence of  blow-ups desingularizing $\sep_{0}^{\aleph} = \{S_{1},S_{2}\}$. 
We will produce a symmetric logarithmic model, with rational residues,   in which $\sep_{0}^{\aleph}$  are separatrices at level 0 of a q-dicritical duplet, corresponding to
a dicritical structure determined by $\pi$, where $\sigma_{1}$ is a non-dicritical blow-up and $\sigma_{2}$
is a dicritical one.

\begin{figure}[h]
\begin{center}
\includegraphics[width=12cm]{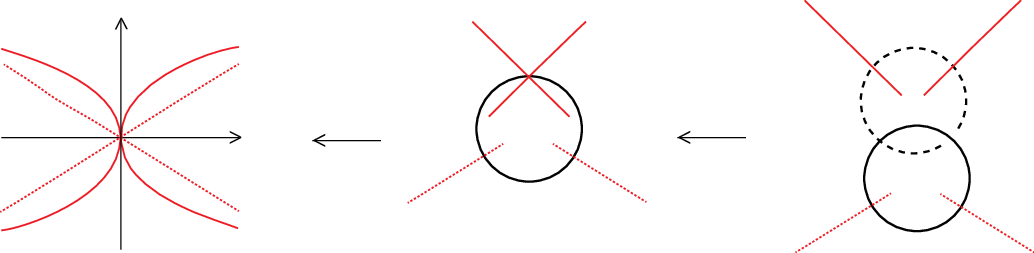}
\put(-260,73){\scriptsize {$S_{1}$}}
\put(-355,73){\scriptsize {$S_{2}$}}
\put(-263,20){\scriptsize {$S_{3}$}}
\put(-352,20){\scriptsize {$S_{4}$}}
\put(-230,43){\scriptsize {$\sigma_{1}$}}
\put(-110,43){\scriptsize {$\sigma_{2}$}}
\put(-170,13){\scriptsize {$D_{1}$}}
\put(-170,50){\scriptsize {$p$}}
\put(-43,-3){\scriptsize {$D_{1}$}}
\put(-43,73){\scriptsize {$D_{2}$}}
\put(-145,73){\scriptsize {$S_{1}$}}
\put(-198,73){\scriptsize {$S_{2}$}}
\put(-130,23){\scriptsize {$S_{3}$}}
\put(-213,23){\scriptsize {$S_{4}$}}
\put(-2,81){\scriptsize {$S_{1}$}}
\put(-88,81){\scriptsize {$S_{2}$}}
\put(-03,5){\scriptsize {$S_{3}$}}
\put(-86,5){\scriptsize {$S_{4}$}}
\end{center}
\end{figure}

 Denote by $D_{1}$ and $D_{2}$ the components of the exceptional divisor introduced by, respectively, $\sigma_{1}$ and $\sigma_{2}$.
In order that (S.2) is satisfied by $D_{1}$ at the final level, we introduce the pair non-real conjugate branches (represented in the picture by   dotted lines)
\[ S_{3}: f_{3} = y- \ic x =0 \quad \text{and} \quad  S_{3}: f_{4} = y+ \ic x =0 ,\]
so that $\sep_{0} = \{S_{1},S_{2},S_{3},S_{4}\}$ are separatrices at level 0 of a dicritical duplet.
We shall now work at level 1. Let $p \in D_{1}$ be the center of $\sigma_{2}$. If we set $\lambda_{D_{1}} = 1$,
in order to have a zero weighted balance of residues at $p$, we can put, for instance
$\lambda_{S_{1}}= \lambda_{S_{2}}= -1/2$.
Let $q_{3}$ and $q_{4}$ denote the points of intersection of $S_{3}$ and $S_{4}$ with $D_{1}$.
Since $i_{p}(D_{1}) = 1$ and $c(D_{1})=-1$, in order to have (I.3) satisfied on $D_{1}$, we can put $i_{q_{3}}(D_{1}) = i_{q_{4}}(D_{1}) = -1$. Since $\sigma_{2}$ is trivial
over $q_{3}$ and $q_{4}$, we have, by (I.4),  $i_{q_{3}}(S_{3}) = i_{q_{4}}(S_{4}) = -1$.
By the consistency condition of Definition \ref{def-consistent}, we then find $\lambda_{S_{3}}= \lambda_{S_{4}}= 1$.
The residues obtained are such that

\[ \eta= - \frac{1}{2} \frac{df_{1}}{f_{1}} - \frac{1}{2} \frac{df_{2}}{f_{2}} +
\frac{df_{3}}{f_{3}} + \frac{df_{4}}{f_{4}}
= \frac{1}{2} \frac{dh}{h}  , \quad
\text{where} \quad h = \frac{f_{3}^{2}f_{4}^{2}}{f_{1}f_{2}}  ,
\]
is a logarithmic model for the data described above.
Let us write $\sigma_{1}$ in coordinates $(u,y)$ such that $x=uy$ and, then,
the blow-up $\sigma_{2}$ at $p$   in coordinates $(u,t)$ such that $y= tu$.
In these coordinates, at level 2, we have  $S_{1}: t-1=0$, $S_{2}: t+1=0$
and $D_{1}: t=0$.
The escape function corresponding to $\eta_{1}= \sigma_{1}^{*} \eta$ at $p$ is
\[R_{\eta_{1}} = \frac{1}{t} + \frac{-1/2}{t-1}+ \frac{-1/2}{t+1} = \frac{-1}{t(t+1)(t-1)} = \frac{P_{\eta_{1}}}{Q_{\eta_{1}}}. \]
Since  $\deg(P_{\eta_{1}}) < 1$,  $\eta_{2} = \sigma_{2}^{*} \eta_{1} = \pi^{*} \eta$ has a point of tangency with $D_{2}$ at the point corresponding to $t= \infty$.
The level $-1$ of $h$ has equation
\[f_{3}^{2}f_{4}^{2}+f_{1}f_{2} = x^{2}(x^{2} + 2y^2 + 1) .\]
That is, the $y$-axis  is a fiber of multiplicity two, thereby contained in $\sing(\eta)$. This one-dimensional
component accounts for the point of tangency detected. It is not a escape separatrix, since it does not force the introduction of additional blow-ups in the
desingularization process.

\begin{figure}[h]
\begin{center}
\includegraphics[height=4cm]{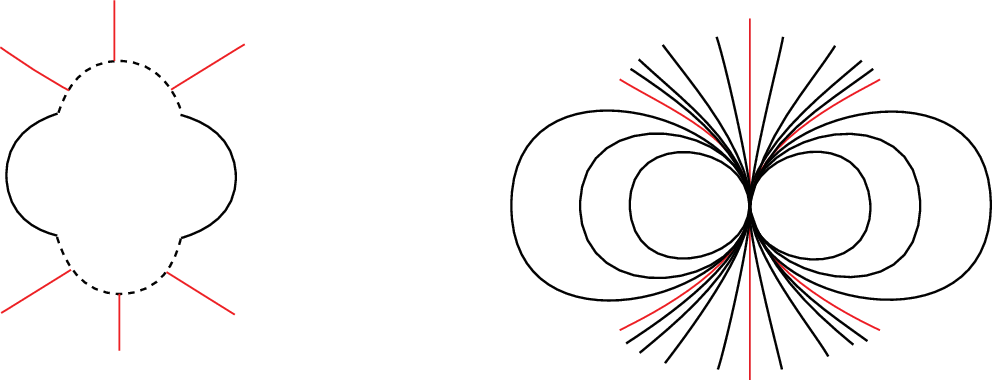}
\put(-220,104){\scriptsize $S_{1}$}
\put(-310,104){\scriptsize $S_{2}$}
\put(-222,11){\scriptsize $S_{2}$}
\put(-308,11){\scriptsize $S_{1}$}
\put(-253,97){\scriptsize $E_{1}$}
\put(-253,17){\scriptsize $E_{2}$}
\end{center}
\end{figure}
The real figure of this example corresponds to an $\ell$-analytic sectorial model   patterned on the infinitesimal
sentence
$\cl{W}=W_{1}W_{2}$, whose words are the infinitesimal semiclasses $W_{1} = \hat{\kappa}_{E_{1}}$ and
$W_{2} = \hat{\kappa}_{E_{2}}$. The semiclasses $\hat{\kappa}_{E_{i}}$, $i=1,2$, yield two parabolic sectors, while the transitions
$\hat{\kappa}_{E_{i}}  \leftrightarrow \hat{\kappa}_{E_{j}}$, $i \neq j$,   two elliptic sectors.
}\end{example}

\begin{example}
{\rm Take, for $i=1,2,3,4$,  the germs  of cuspids at $(\C^{2},0)$ defined by the equations
\[ S_{i}: f_{i}(x,y) = x^3 - (y-a_{i}x)^2 =0, \]
where $a_{1} = 1$, $a_{2} = -1$, $a_{3} = 4$, $a_{4} = -4$.
We consider $\sep_{0} = \{S_{1},S_{2},S_{3},S_{4}\}$ as the initial level  of a configuration of separatrices framed on a
dicritical structure determined by the sequence of blow-ups that desingularizes $\sep_{0}$, in which only the first blow-up
is dicritical.
We produce a logarithmic model with these data having rational residues.

\begin{figure}[h]
\begin{center}
\includegraphics[width=8cm]{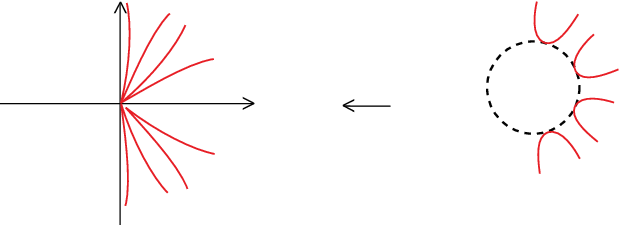}
\put(-155,68){\scriptsize {$S_{1}$}}
\put(-177,82){\scriptsize {$S_{3}$}}
\put(-155,17){\scriptsize {$S_{2}$}}
\put(-177,3){\scriptsize {$S_{4}$}}
\put(-5,65){\scriptsize {$S_{1}$}}
\put(-25,82){\scriptsize {$S_{3}$}}
\put(-5,33){\scriptsize {$S_{2}$}}
\put(-23,14){\scriptsize {$S_{4}$}}
\put(-94,49){\scriptsize {$\sigma_{1}$}}
\put(-58,33){\scriptsize {$D_{1}$}}
\end{center}
\end{figure}

In order to zero the weighted balance of residues
at $0 \in \C^{2}$, we can take, for instance, $\lambda_{1} = \lambda_{3} = 1/10$ and $\lambda_{2} = \lambda_{4} = -1/10$. We then  set
\[ \eta = \sum_{i=1}^{4} \lambda_{i} \frac{df_{i}}{f_{i}} = \frac{1}{10} \frac{dh}{h}, \quad \text{where}\ \quad  h = \frac{f_{1} f_{3}}{f_{2} f_{4}}.\]
Denoting by $\sigma_{1}: (\tilde{\C^{2}},D_{1}) \to (\C^{2},0)$  the dicritical   blow-up at $0 \in \C^{2}$ and considering  coordinates
$(x,t)$ such that $y=tx$, at level 1, each   $S_{i}$ has  $x  - (t-a_{i})^2 =0$ as equation.
Hence, in these coordinates, we get the escape function
\[ R_{\eta} = \frac{1}{10} \left( \frac{1}{t-1} + \frac{-1}{t+1} + \frac{1}{t-4} + \frac{-1}{t+4} \right) = \frac{t^2 -4}{(t^{2}-1)(t^{2}-16)}, \]
which has two affine zeroes, $t = 2$ and $t = -2$.
The curves of equations
\[4(t-2)^4-9(t-2)^3-8(t-2)^2x-18(t-2)^2+9(t-2)x+4x^2-18x = 0\]
and
\[-4(t+2)^4-9(t+2)^3+8(t+2)^2x+18(t+2)^2+9(t+2)x-4x^2+18x = 0\]
are invariant by $\eta_{1} = \sigma^{*} \eta$ and are tangent to the divisor $D_{1}$ at, respectively, $t=2$ and $t=-2$
(they correspond  to the level curves $h= 1/81$ and $h=81$).
Their local branches at these points are, thus, escape separatrices.

Now, let us consider a set of parameters of the form $\tau = \{(2+a,\lambda),(-2-a,-\lambda)$, where $a,\lambda \in \R^{*}$ are small, to be determined
following the guidelines of the proof of Proposition \ref{prof-no-real-roots}. We consider the balanced symmetric logarithmic modification
with parameters $\tau$ given, for instance, by
\[\eta^{\tau} = \eta + \lambda \frac{d g_{1}}{g_{1}} - \lambda \frac{d g_{2}}{g_{2}}, \]
where $g_{1} = y - (2+a)x$ and  $g_{2} = y + (2+a)x$.
Its escape function is
\[  R_{\eta^{\tau}} =  R_{\eta} + \frac{\lambda}{t-(2+a)} + \frac{-\lambda}{t+(2+a)}=  \frac{t^2 -4}{(t^{2}-1)(t^{2}-16)}  + \frac{2 \lambda(2+a)}{t^{2}-(2+a)^{2}}
= \frac{ P_{\eta^{\tau}}}{ Q_{\eta^{\tau}}} ,\]
whose numerator is the biquadratic  polynomial
\[ P_{\eta^{\tau}}(t) = (1 +2\lambda (2+a))t^4 +(-(2+a)^2 - 4 -34(2+a)\lambda)t^2 + (4(2+a)^2 + + 32 (2+a)\lambda).\]
The discriminant of the corresponding  quadratic polynomial obtained by setting $t^{2} = z$ is
\[ \Delta = (2+a)^4 + 36 (2+a)^3 \lambda + 900 (2+a)^2 \lambda^2 -8(2+a)^2 -144(2+a)\lambda + 16 .\]
Taking, for instance, $a=1/10$ and $\lambda = -1/10$ we get $\Delta = -237215/10000 <0$.
This means that  $P_{\eta^{\tau}}(t)$ has no real roots and, thus, $\eta^{\tau}$ has no real escape points.
Using these values, we have
$g_{1} = y - (2+1/10)x$ and $g_{2} = y + (2+1/10)x$, so that
\[\eta^{\tau} = \eta - \frac{1}{10} \frac{d g_{1}}{g_{1}} +  \frac{1}{10} \frac{d g_{2}}{g_{2}} =
\frac{1}{10} \frac{dh_{1}}{h_{1}}, \quad \text{where} \quad   h_{1} = h \frac{g_{2}}{g_{1}} =   h = \frac{f_{1} f_{3}g_{2}}{f_{2} f_{4}g_{1}}.\]
We invitine the reader to draw the  figure corresponding to  the  sectorial decomposition of this example and also to try to devise other examples by himself.
}\end{example}

Our method, as  depicted in the two examples above, allows the construction of explicit examples of real meromorphic
(or real rational) functions whose level sets satisfy prescribed $\ell$-sectorial models. These level set are real analytic
(or real algebraic) curves that, regarding their portions on the elliptic sectors as petals, may  have a pleasant flower-like shape.

\bibliographystyle{plain}
\bibliography{referencias}

\end{document}